\documentclass[11pt]{article}%
\usepackage{float} % fix table inside context 
\usepackage{fullpage}
\usepackage{amsfonts}
\usepackage{amsmath}
\usepackage{amssymb}
\usepackage{amsbsy}
\usepackage{amsthm}
\usepackage{mathtools}
\usepackage{epsfig}
\usepackage{graphicx}
\usepackage{colordvi}
\usepackage{graphics}
\usepackage{color}
\usepackage{bm}
\usepackage{verbatim} 
\numberwithin{equation}{section}

\newtheorem{thm}{Theorem}[section]

\newtheorem{lemma}[thm]{Lemma}
\newtheorem{dfn}[thm]{Definition}

\newtheorem{alg}[thm]{Algorithm}

\newtheorem{assumption}[thm]{Assumption}

\begin{document}

\title{Superlinear convergence of Anderson accelerated Newton's method for solving stationary Navier-Stokes equations}

\author{
Mengying Xiao 
\thanks{Department of Mathematics and statistics, University of West Florida, Pensacosa, FL 32514 (mxiao@uwf.edu).}
}
\date{}
%\subjclass[2010]{65H10}
%\keywords{
%Anderson acceleration,
%fixed-point iterations,
%nonlinear iterations,
%convergence theory,
%damping}

\maketitle

\begin{abstract}
This paper studies the performance Newton's iteration applied with Anderson acceleration for solving the incompressible steady Navier-Stokes equations. We manifest that this method converges superlinearly with a good initial guess, and moreover, a large Anderson depth decelerates the convergence speed comparing to a small Anderson depth. We observe that the numerical tests confirm these analytical convergence results, and in addition, Anderson acceleration sometimes enlarges the domain of convergence for Newton's method.
\end{abstract}

\section{Introduction}
This work studies the performance of Newton's method with an acceleration technique, known as Anderson acceleration introduced by \cite{Anderson65}, for solving the incompressible Navier-Stokes equations. It is inspired by the work of \cite{PRX19, EPRX20, PS20}, where \cite{PRX19} shows how Anderson acceleration can locally improve the convergence rate of the linearly converging Picard iteration for solving the NSE, the numerical tests from \cite{EPRX20} shows that Anderson acceleration slows the convergence speed of Newton's method but enlarges the domain of convergence for NSE, and \cite{PS20} justifies, both theoretically and numerically, superlinear convergence of this method with depth $m=1$ for several benchmark nonlinear problems. Our work manifests the superlinear convergence of Anderson accelerated Newton's method (AAN) for solving steady Navier-Stokes problem analytically and numerically.

Navier-Stokes equations (NSE) are governed by the following:
\begin{align}
 - \nu \Delta u + u\cdot \nabla u + \nabla p & = f, \label{nse1t}\\
\nabla \cdot u & = 0, \label{nse2t}
\end{align}
 on a domain $\Omega\subset \mathbb{R}^d\ (d= 2,\ 3)$,
with appropriate boundary conditions.  Here $u$ is velocity, $p$ is pressure, $f$ is external force such as buoyancy, gravity, etc. $\nu$ is the kinetic viscosity, where its reciprocal is known as Reynolds number $Re$. 
The Newton's method of NSE takes the form: Given $u_0$, find $(u_k, p_k)$ satisfying 
\begin{align*}
-\nu \Delta u_k + u_k \nabla \cdot u_{k-1}  + u_{k-1} \nabla \cdot u_k - u_{k-1}\nabla \cdot u_{k-1} + \nabla p_k & = f,\\
\nabla \cdot u_k & = 0.
\end{align*}
It is known that this iterative method is a local method and converges quadratically if the initial guess is good enough. Our work herein demonstrates the Newton's method applied with Anderson acceleration with general depth $m$ for solving NSE converges superlinearly both theoretically and numerically, if the initial guess is good enough. Moreover, a large depth decelerates the convergence speed due to more high order terms appeared in the one-step bound, comparing to a small depth.

This paper is organized as follows. In section 2, we provide notations, mathematical preliminary, finite element scheme for steady NSE, and then present some properties of the solution operator to Newton's method. Section 3 gives the algorithms and convergence results of Anderson accelerated Newton's method with varying depth for solving steady NSE. Several numerical tests are provided in section 4 that confirm our analytical results.

\section{Notation and Mathematical preliminaries}
This section provides notation, mathematical preliminaries, and background to allow for a smooth analysis in later sections. First, we will give function spaces and notational details, followed by finite element discretization preliminaries. Then we provide the Newton's iteration of steady Navier-Stokes equations and some basic properties. Throughout this paper, we consider homogeneous Dirichlet boundary condition for velocity $$ u = 0 \text{ on } \partial\Omega.$$  Of course, all results can be extended to other common boundary conditions with some extra work and these are not discussed here.
\subsection{Discretization of NSE}
The domain $\Omega\subset \mathbb{R}^d$ ($d = 2, 3$) is assumed to be simply connected and to either be a convex polytope or have a smooth boundary . The $L^2(\Omega)$ norm and inner product will be denoted by $\| \cdot \|$ and $(\cdot , \cdot)$, respectively, and all other norms will be labeled with subscripts. The natural function spaces for velocity and pressure are given by 
\begin{align*}
X& =  H_0^1(\Omega)^d : = \{ v\in L^2(\Omega)^d, \ \nabla v \in L^2(\Omega)^{d\times d}, v = 0 \text{ on } \partial\Omega\},\\
Q & = L_0^2(\Omega) :=\{ q\in L^2(\Omega), \int_\Omega q\  dx = 0 \}.
\end{align*}
The Poincar\'e inequality is known to hold in $X$ \cite{laytonbook}: there exists $C_P>0$ dependent only on the domain $\Omega$ satisfying 
$$ \| v\| \le C_P \| \nabla v\|,$$
for any $v\in X$.
\begin{dfn}
Define a trilinear form: $b:X\times X \times X \to \mathbb{R}$ such that for any $u,v,w\in X$
$$ b(u,v,w):= \frac12 ( (u\cdot \nabla v,w) - (u\cdot \nabla w, v)).$$
\end{dfn}
The operator $b$ is skew-symmetric
\begin{align}
b(u,v,v) = 0,
\label{bss}
\end{align}
and satisfies inequality
\begin{align}
b(u,v,w) \le M \|\nabla u\| \|\nabla v\| \|\nabla w\|,
\label{bbd}
\end{align}
for any $u,v,w\in X$, with $M$ depending only on $|\Omega|$, see \cite{laytonbook}.

We will denote by $\tau_h$ a regular, conforming triangulation of $\Omega$ with maximum element diameter $h$. The finite element spaces will be denoted as $X_h \subset X, \ Q_h\subset Q$, and we require that $(X_h,Q_h)$ pair satisfies the usual discrete inf-sup condition \cite{laytonbook}. For example one could select Taylor-Hood elements, 
Scott-Vogelius elements on an appropriate mesh 
\cite{arnold:qin:scott:vogelius:2D,MR2519595,zhang:scott:vogelius:3D}, 
or the mini element \cite{ABF84}, etc.%so that the discretized NSE scheme is well-posed. 

The discrete stationary NSE is given by: Finding $(u,p) \in (X_h, Q_h)$  such that  for and $(v,q)\in (X_h, Q_h)$
\begin{align}
 b(u,u,v) + \nu(\nabla u, \nabla v) - (p, \nabla \cdot v) & = (f,v), \label{nse1} \\
 (\nabla \cdot u, q) & = 0. \label{nse2} 
\end{align}
Let $V_h$ be the discretely divergence-free subspace as  
$$V_h := \{ v\in X_h,  (\nabla \cdot v, q) = 0, \ \forall q\in Q_h\},$$
then an equivalent formulation of \eqref{nse1}-\eqref{nse2} is obtained: Find $u\in V_h$  such that  for any $v\in V_h$ \begin{align}
b(u,u,v) + \nu (\nabla u, \nabla v) & = (f,v) . \label{nse3}
\end{align}
This paper focuses on studying the convergence behavior after applying Anderson acceleration to the Newton's method.  We assume that systems \eqref{nse1}-\eqref{nse2} and \eqref{nse3} are well-posed for simplicity. In other words, the small data condition 
\begin{align}
\label{smalldata}
\kappa : = \nu^{-2}M\|f\|_{-1}<1,
\end{align} is satisfied, see \cite{laytonbook}.  
However, all results presented here can be extended to the case where the discretized steady NSE has distinct solutions ($\kappa \ge1$). With deflation techniques \cite{BG71,FBF15}, one may find out distinct solutions to steady NSE model using nonlinear solvers. For the rest of the paper, we assume \eqref{smalldata} holds.

\subsection{Properties of Newton's solution operator for steady NSE}
In this subsection, we define the Newton's solution operator for steady NSE and present some properties of it. 
\begin{dfn} Given $u\in V_h$, define a mapping $G: V_h \to V_h$ satisfying \begin{align}
\label{eqn:newt1} 
b(u, G(u), v) + b(G(u), u,v) - b(u,u,v) + \nu (\nabla G(u), \nabla v) & = (f,v),
\end{align}
for any $v\in V_h$ and called $G$ the Newton solution operator.
\end{dfn}
Next lemma shows that this operator $G$ is well-defined on a ball. 
\begin{lemma}
\label{lemma:range}
Equation \eqref{eqn:newt1} has a unique solution for any $u\in B(0,R):= \{v\in  V_h\mid \|\nabla v\| \le R\}$ with $R< \nu M^{-1}.$ Moreover, the following inequalities hold
\begin{align}
\label{eqn:gubd}
\|\nabla G(u)\| \le   C_0R^2 + (1+C_0R)\nu^{-1}\|f\|_{-1} : = R_G,
\end{align}
\begin{align}
\label{eqn:guubd}
\|\nabla (G(u)-u)\| \le 2(1+C_0R)^{1/2}R + 2\nu^{-1}(1+C_0R)\|f\|_{-1} := R_{res},
\end{align}
where $C_0 : = \frac{\nu^{-1}M}{1-\nu^{-1}MR}.$
\end{lemma}
%\textcolor{blue}{
%$
%R_G, R_{res} \sim \nu, |\Omega|, f, R.
%$
%}
\begin{proof}
We begin the proof by finding an upper bound of $\|\nabla G(u)\|.$ Setting $v = G(u)$ in \eqref{eqn:newt1} eliminates the first trilinear form and yields 
\begin{align*}
\nu \| \nabla G(u)\|^2 & = - b(G(u), u, G(u)) + b(u, u,G(u)) + (f, G(u))\\
& \le M \|\nabla G(u)\|^2\| \nabla u\| + M \|\nabla u\|^2 \|\nabla G(u)\| + \|f\|_{-1}\|\nabla G(u)\|,
\end{align*}
thanks to \eqref{bss}, \eqref{bbd} and Cauchy-Schwarz inequality. Dividing both sides by $\|\nabla G(u)\|$, it reduces to 
\begin{align*}
\nu (1-\nu^{-1}M\|\nabla u\|)\|\nabla G(u)\| \le M\|\nabla u\|^2 + \|f\|_{-1}.
\end{align*}
Thus from the assumptions $\|\nabla u\| \le R< \nu M^{-1}$ and \eqref{smalldata}, we have \eqref{eqn:gubd}.
Since the system \eqref{eqn:newt1} is linear and finite dimensional, \eqref{eqn:gubd} is sufficient to imply solution uniqueness and therefore existence. 

Then we show inequality \eqref{eqn:guubd}. Setting $v = G(u)-u$ in equation \eqref{eqn:newt1} gives
\begin{align*}
b(G(u),u, G(u)-u) + \nu(\nabla G(u), \nabla (G(u)-u)) = (f,G(u)-u),
\end{align*}
thanks to \eqref{bss}. 
Applying the polarization identity, \eqref{bbd} and Cauchy-Schwarz inequality, we have
%\begin{align*}
%\frac{\nu}2 \left(\|\nabla G(u)\|^2 + \|\nabla (G(u)-u)\|^2 - \|\nabla u\|^2 \right)  & = - b(G(u), u, G(u)) + (f, G(u) -u)  \\
%& \le MR\|\nabla 
%\end{align*}?
\begin{align*}
\frac{\nu}2(1-\nu^{-1}MR) \left( \|\nabla G(u)\|^2 + \|\nabla (G(u)-u)\|^2 \right) \le \frac{\nu}2 \|\nabla u\|^2+  \|f\|_{-1}\|\nabla (G(u)-u)\|,
\end{align*}
Dropping the term with $\|\nabla G(u)\|^2$ and utilizing the Young's inequality, we obtain 
\begin{align*}
\|\nabla (G(u)-u)\|^2 \le 2(1-\nu^{-1}MR)^{-1}\|\nabla u\|^2 + 4 \nu^{-2}(1-\nu^{-1}MR)^{-2}\|f\|_{-1}^2.
\end{align*}
From the definition of $C_0$ and identity $(1 -\nu^{-1}MR)^{-1} = 1+C_0R$, \eqref{eqn:guubd} is achieved.
\end{proof}
Now we state the Newton's iteration for steady NSE and a few  properties of $G$ are followed. 
\begin{alg}[Newton's iteration for steady NSE]
\label{alg:newt}
The Newotn's method for steady Navier-Stokes equations is as below:
\begin{enumerate}
\item[Step 0] Give $w_0\in B(0,R)$.
\item[Step k] Compute $w_k = G(w_{k-1})$.
\end{enumerate}
\end{alg}
Obviously, Algorithm \ref{alg:newt} fails when ever $w_k \not\in B(0,R)$ for some integer $k$. In order to discuss its convergence order, we make an assumption:
%
%produces an infinite sequence $\{ w_k \}_{k =0}^\infty$ when $w_j\in B(0,R)$ for all $j\in \mathbb{N}$. (Otherwise it fails at step $k$ whenever $w_k \not\in B(0,R)$.) We make the following assumption.
\begin{assumption}
\label{assump00}
Assume the sequence $\{w_k\}$ from Algorithm \ref{alg:newt} satisfies $$w_k \in B(0,R):= \{v\in  V_h\mid \|\nabla v\| \le R\}$$ for all $k \in \mathbb{N}$, where $R < \nu M^{-1}$.
\end{assumption}
It is well-known that Algorithm \ref{alg:newt} converges quadratically \cite{GR86}. Here we present a different point of view to the quadratical convergence, expressing the error bound in terms of residuals $G(u) -u$, the following lemma will be used multiple times in the next section.
\begin{lemma}
\label{lemma:newtquad}
Assume $u,w \in B(0,R)$ with $R< \nu M^{-1}$, thus 
\begin{align}
\|\nabla (G(w)-G(u))\|  & \le 2C_0\|\nabla (w-u)\|\|\nabla (G(u)-u)\| + C_0\|\nabla (w-u)\|^2.
\label{eqn:errbd}
\end{align}
where $C_0$ is defined in Lemma \ref{lemma:range}.
\end{lemma}
\begin{proof}
For any $u,w\in B(0,R)$,we  rewrite equation \eqref{eqn:newt1} as 
\begin{align*}
b(u,G(u),v) + b(G(u)-u, u, v) + \nu (\nabla G(u),\nabla v) & = (f,v) ,  \\
b(w, G(w),v) + b(G(w)-w, w,v) + \nu (\nabla G(w), \nabla v) & = (f,v), 
\end{align*}
for any $v\in V_h$.
Subtracting the above two equations gives
\begin{multline}
\label{eqn3}
b(w,G(w)-G(u),v) + b(G(w)-G(u),w,v)+ b(w-u,G(u) -u ,v)   \\ + b(G(u)-w,w-u,v)
+ \nu (\nabla (G(w)-G(u)), \nabla v) = 0.
\end{multline}
Setting $v = G(w) - G(u)$ eliminates the first term and produces 
\begin{align*}
\nu(1-\nu^{-1}M\|\nabla w\|)\|\nabla (G(w)-G(u))\| \le  M\|\nabla (w-u)\| (\|\nabla (G(u)-w)\| + \|\nabla (G(u)-u)\|),
\end{align*}
thanks to inequality \eqref{bbd}. Then we have
\begin{align}
\label{eqn:newterr}
\|\nabla (G(w) - G(u))\| \le C_0  \|\nabla (w-u)\| (\|\nabla (G(u) - u)\| + \|\nabla (G(u)-w)\| ),
\end{align}
due to assumptions $\|\nabla w\| \le R < \nu M^{-1} $ and \eqref{smalldata}. 
From triangle inequality, we have \eqref{eqn:errbd} and finish the proof.
\end{proof}
Easily, one can end up with inequality
\begin{align*}
\|\nabla (w_{k+1} - w_k)\| \le C_0\|\nabla(w_k - w_{k-1})\|^2,
\end{align*} 
by setting  $u  = w_{k-1}, \ w= w_k := G(w_{k-1})$ in \eqref{eqn:newterr}.
%\textcolor{blue}{
%Xiao says: do I want to delete this part?\\
%Furthermore, the sequence $\{w_k\}$ converges when
%\begin{align}
%\label{eqn:doc}
% w_0 \in D := \{v\in B(0,R)\subset V_h\mid \|\nabla (G(v)-v)\| < C_0^{-1}\}.
% \end{align}
% The set $D$ is called the domain of convergence for Algorithm \ref{alg:newt}. Note that the larger $R$ is, the smaller domain of convergence for Newton's method, due to $C_0^{-1} = (\nu^{-1}M)^{-1} - R.$
%}

Lastly, we show that the solution operator $G$ is Fr\'echet differentiable. 
\begin{dfn}
\label{def:dg}
Given $u, u+h \in B(0,R)$ with $R< \nu M^{-1}$, define $G'(u;\cdot ): V_h \to V_h$ such that 
\begin{multline}
\label{eqn:dg}
b(u, G'(u;h),v) + b(G'(u;h), u,v) + b(h, G(u) - u, v) + b(G(u)-(u+h), h,v)  \\
+ \nu(\nabla G'(u;h), \nabla v) = 0,
\end{multline}
for any $v\in V_h$.
\end{dfn}
\begin{lemma}
\label{lemma:dg}
$G'$ is well-defined over $B(0,R).$ Moreover, $G'$ is the Fr\'echet derivative of $G$ satisfying
\begin{align}
\label{eqn:fdg} 
\|\nabla (G(u+h) - G(u) - G'(u;h))\| \le  2C_0 \|\nabla h\| \|\nabla (G(u+h)-G(u))\|,
\end{align}
and is bounded by 
\begin{align}
\label{eqn:dgbd}
\|\nabla G'(u;h)\| &\le 2 C_0 \| \nabla (G(u)-u)\|  \|\nabla h\| + C_0 \|\nabla h\|^2, 
\end{align} 
for any $u, u+h \in B(0,R).$
\end{lemma}
\begin{proof}
This proof includes two parts. First, we show $G'$ is well-defined and has an upper bound. Setting $v = G'(u;h)$ in \eqref{eqn:dg} eliminates the first term and produces
\begin{align*}
\nu (1-\nu^{-1}M\|\nabla u\|) \|\nabla G'(u;h)\| \le   2M \|\nabla (G(u)-u) \| \|\nabla h\| + M\|\nabla h\|^2,
\end{align*}
thanks to \eqref{bbd}, which reduces to \eqref{eqn:dgbd}.
 Since system \eqref{eqn:dg} is linear and finite dimensional, \eqref{eqn:dgbd} guarantee $G'$ is well-defined. 
Second, we manifest $G'$ is the Fr\'echet derivative of $G$. Denoting $\xi = G(u+h) - G(u) - G'(u;h)$ and subtracting \eqref{eqn:dg} from \eqref{eqn3} with $w = u+h$, we have
\begin{align*}
b(u, \xi, v)  + b(h,G(u+h)-G(u),\xi) + b(\xi, u,v) + b(G(u+h)-G(u), h,\xi) + \nu (\nabla \xi, \nabla v) = 0.
\end{align*}
Setting $v = \xi$ eliminates the first term and yields 
\begin{align}
\label{eqn:0}
 \nu(1-\nu^{-1}M\|\nabla u\|)  \| \nabla \xi \| \le 2M  \| \nabla (G(u+h)-G(u))\| \|\nabla h \| ,
 \end{align}
due to \eqref{bbd}.
Thus from $\| \nabla u\| \le R < \nu M^{-1} $ and \eqref{smalldata}, it leads to \eqref{eqn:fdg}.
%Since $\|\nabla \xi\| /\|\nabla h\| \to 0$ as $\|\nabla h\| \to 0$ due to Lemma \ref{lemma:newtquad}. 
Therefore $G'$ is the Fr\'echet derivative of $G$ and we finish the proof.
\end{proof}
\section{Anderson accelerated Newton's iteration for steady NSE}
In this section, we state algorithms of Anderson accelerated Newton's iteration (AAN) for steady NSE and give an analysis of superlinear convergence order. We start with the simplest case where the Anderson depth is 1 in the coming subsection, and then move on to  the Anderson  depth $m=2$ and general depth cases in the second and third subsections respectively.
\subsection{Anderson accelerated Newton's method with depth $m =1$}
\begin{alg}[Anderson accelerated Newton's iteration with depth $m = 1$ (AAN m=1)] 
\label{alg:aa1}
The algorithm of Anderson accelerated Newton's method with depth $m=1$ is stated as below: 
\begin{enumerate}
\item[Step 0] Guess $u_0\in B(0,R)$.
\item[Step 1] Compute $\tilde u_1 = G(u_0)$ and set the residual $y_1 =\tilde  u_1 - u_0,$ update $\ u_1 = \tilde u_1$.
\item[Step k] For $k = 2,3,\dots$
\begin{enumerate}
\item[a)] Compute $\tilde u_k = G(u_{k-1})$ and set the residual $y_k = \tilde u_k - u_{k-1}.$
\item[b)] Find $\alpha_k \in \mathbb{R}$ minimizing 
$$ \|\nabla ((1-\alpha_k) y_k + \alpha_k y_{k-1})\| .$$%= \|\nabla (y_k - \alpha_k (y_k - y_{k-1}))\|. $$
\item[c)] Update $u_k = (1-\alpha_k) \tilde u_k + \alpha_k \tilde u_{k-1}.$
\end{enumerate}
\end{enumerate}
\end{alg}
We will use the residual sequence $\{y_k\}$ to discuss the convergence behavior of Algorithm \ref{alg:aa1}. For smooth analysis, we use the following notation throughout this subsection 
$$ e_k = u_ k - u_{k-1}, \quad \tilde e_k = \tilde u_k - \tilde u_{k-1}, \quad y_k^\alpha = (1-\alpha_k) y_k + \alpha_k y_{k-1}. $$
Comparing Algorithm \ref{alg:aa1} with the usual Newton's Algorithm \ref{alg:newt}, we add a minimization step at each iteration. It is clear that Algorithm \ref{alg:aa1} is back to Algorithm \ref{alg:newt} at step $k$ whenever $\alpha_k = 0$.
We make the following assumption in order to study the behavior of Algorithm \ref{alg:aa1}.
\begin{assumption}
\label{assump01}
For step $k\ge 2$, assuming $\alpha_k \neq 0$ and $ u_j \in B(0,R)$ with $R< \nu M^{-1}$  for all $j \le k.$ 
\end{assumption}
We now give an expression of $\alpha_k$ in terms of residuals in next lemma. 
\begin{lemma}
\label{lemma:alpha1}
For any step k with $\alpha_k \neq 0$, let Anderson gain $\theta_k : = \|\nabla y_k^\alpha \| / \|\nabla y_k\|$, 
  then $\theta_k \in [0,1)$ and 
 \begin{align}
 \label{eqn:alpha0}
 |\alpha_k| &= \frac{\sqrt{1-\theta_k^2}\|\nabla y_k\|}{\|\nabla (y_k - y_{k-1})\|}.
 %\\
% \label{eqn:alpha1}
% |1-\alpha_k| &= \frac{\sqrt{\|\nabla y_{k-1}\|^2 - \theta_k^2 \|\nabla y_k\|^2}}{\|\nabla (y_k - y_{k-1})\|}
 \end{align}
% consequently, $\theta_k \|\nabla y_k \| \le \|\nabla y_{k-1}\|$%, which implies $\|\nabla (y_k - y_{k-1})\| \ge (1-\theta_k)\|\nabla y_k\|$ and $|\alpha_k| \le \frac{\sqrt{1-\theta_k^2}}{1-\theta_k}$.
\end{lemma}
\begin{proof}
It is clearly $\theta_k \in (0,1]$ by the choice of $\alpha_k$.
Let $f(\alpha) = \| \nabla (1-\alpha) y_k  + \alpha y_{k-1}\|^2$, then $\alpha_k$ is the stationary point of $f(\alpha).$ That is, $f'(\alpha_k) = 0,$ which gives $$ \alpha_k \| \nabla (y_k - y_{k-1})\|^2 =  (\nabla y_k , \nabla (y_k - y_{k-1})).$$
By the definition of $\theta_k$, we have 
\begin{align*}
\theta_k^2 \|\nabla y_k\|^2  & = \|\nabla y_k^\alpha\|^2 
 = \|\nabla y_k\|^2 - \alpha_k^2 \|\nabla (y_k - y_{k-1})\|^2, 
\end{align*}
which leads to \eqref{eqn:alpha0}.
\end{proof}
Before we give the main convergence result for Algorithm \ref{alg:aa1}, we present a few lemmas that play key roles in the analysis. First, we list a few identities that will be used repeatedly in this subsection
\begin{align}
\label{eqn:id11}
\tilde u_k - u_k & = \alpha_k \tilde e_k ,\\
\label{eqn:id12}
y_k- y_{k-1} & = \tilde e_k - e_{k-1},\\
\label{eqn:id13}
y_k^\alpha & = e_k + \alpha_k e_{k-1}.
\end{align} 
Next, we show that the difference between solutions from successive iterations can be bounded by the residuals. 
\begin{lemma}
\label{lemma1}
Let Assumption \ref{assump01} holds, there exists a positive constant $C_1 = C_1(\alpha_k,  \nu, |\Omega|, f, R)$ such that 
\begin{align}
\label{eqn:ek1bd}
|\alpha_k| \|\nabla e_{k-1}\| \le C_1 \| \nabla y_k\|,
\end{align}
and 
\begin{align}
\label{eqn:ekbd}
\|\nabla e_k \| \le C_1 \|\nabla y_k\|.
\end{align}
To be more specific, $C_1 = \sqrt{3+2C_0R +4C_0^2 (R+2|\alpha_k|R_G)^2 +4C_0^2R_{res}^2} .$
\end{lemma}
\begin{proof}
We begin by the following equation 
\begin{align}
\label{eqn:newt}
b(u_j, \tilde u_{j+1}, v) + b(\tilde u_{j+1}, u_j, v)  - b(u_j,u_j,v) + \nu (\nabla \tilde u_{j+1}, \nabla v) & = (f,v),
\end{align}
for any nonnegative integer $j$. Subtracting \eqref{eqn:newt} with $j = k-2$ from \eqref{eqn:newt} with $j = k-1$ yields 
\begin{align}
\label{eqn:temp}
\nu(\nabla \tilde e_k, \nabla v) & =  - b(\tilde e_k, \tilde u_k, v) - b(u_{k-2}, \tilde e_k, v) + b(y_k - y_{k-1}, y_k, v) - b(y_{k-1}, e_{k-1}, v) \nonumber\\
&= -b(\tilde e_k , u_k ,v) -  b( u_{k-2} + \alpha_k\tilde e_k, \tilde e_k, v) + b(y_k - y_{k-1}, y_k, v) - b(y_{k-1}, e_{k-1}, v) ,
\end{align}
thanks to \eqref{eqn:id12} and \eqref{eqn:id11}.
Setting $v = e_{k-1}$ eliminates the last term and gives 
\begin{multline*}
\frac{\nu}2 \left( \|\nabla \tilde e_k\|^2 + \|\nabla e_{k-1}\|^2 - \|\nabla (y_k - y_{k-1})\|^2 \right) \le 
M\|\nabla  u_k\| \|\nabla \tilde e_k \| \|\nabla e_{k-1}\|  \\
 + M\|\nabla ( u_{k-2}+ \alpha_k \tilde e_k)\| \|\nabla \tilde e_k\| \|\nabla (y_k - y_{k-1})\|+ M\|\nabla (y_k - y_{k-1})\| \|\nabla y_k\| \|\nabla e_{k-1}\|,
\end{multline*}
thanks the polarization identity, \eqref{eqn:id12}, \eqref{bbd} and \eqref{bss}. 
From inequalities $ab \le \frac{a^2 + b^2}2$, triangle inequality $\|\nabla (u_{k-2} + \alpha_k \tilde e_k) \| \le R + 2|\alpha_k|R_G, $ and  the Young's inequality, we obtain 
\begin{multline*}
\frac{\nu}4(1-\nu^{-1}MR) \left(\|\nabla \tilde e_k\|^2 + \|\nabla e_{k-1}\|^2\right) \le \\
\left( \frac{\nu}2 +  \nu^{-1}(1-\nu^{-1}MR)^{-1}M^2 (R+2|\alpha_k|R_G)^2 + \nu^{-1}(1-\nu^{-1}MR)^{-1} M^2\|\nabla y_k\|^2  \right)\|\nabla (y_k-y_{k-1})\|^2 .
\end{multline*}
Dropping the term with $\|\nabla \tilde e_k\|$ yields
$$ \|\nabla e_{k-1}\| \le \tilde C \|\nabla (y_k-y_{k-1})\| \le C_1\|\nabla (y_k-y_{k-1})\| ,$$
where $\tilde C = \sqrt{2(1+C_0R) +4C_0^2 (R+2|\alpha_k|R_G)^2 +4C_0^2R_{res}^2} $. 
Applying \eqref{eqn:alpha0} gives
$$ |\alpha_k| \|\nabla e_{k-1}\| \le \tilde C\sqrt{1-\theta_k^2}\|\nabla y_k  \|,$$ and \eqref{eqn:ek1bd}.
Furthermore,  we obtain $$ \| \nabla e_k \| \le \|\nabla y_k^\alpha\| + |\alpha_k| \|\nabla e_{k-1}\| \le (\theta_k + \tilde C\sqrt{1-\theta_k^2}) \|\nabla y_k\| ,$$
thanks to \eqref{eqn:id13}. Therefore \eqref{eqn:ekbd} holds as
$\max\limits_{\theta_k\in [0,1)} \theta_k + \tilde C\sqrt{1-\theta_k^2}  = \sqrt{1+\tilde C^2} = C_1.$
\end{proof}
%}
The next lemma uses the Fr\'echet derivative properties of $G$ as presented in Lemma \ref{lemma:dg}.
\begin{lemma}
\label{lemma2}
Let Assumption \ref{assump01} holds, then 
\begin{align}
\label{eqn:yk2}
  \|\nabla (G'(u_{k-1}; e_k) + \alpha_k G'(u_{k-2}; e_{k-1}))\| 
\le  C_0C_1 \left( 4+C_1+(2+C_1)/|\alpha_k| \right)\|\nabla y_k\|^2 = \mathcal{O}(\|\nabla y_k\|^2).
\end{align}
and \begin{multline}
\| \nabla (y_{k+1} - G'(u_{k-1}; e_k) - \alpha_k G'(u_{k-2}; e_{k-1}))\| \le \\
2C_0^2C_1^2 ( 2+C_1 + 2/|\alpha_k| + C_1/|\alpha_k|^2) \|\nabla y_k \|^3 = \mathcal{O}(\|\nabla y_k\|^3) .
\label{eqn:yk1}
\end{multline}
\end{lemma}
\begin{proof}
Utilizing \eqref{eqn:dgbd} produces
\begin{align*}
\|\nabla G'(u_{k-1}; e_k)\|  & \le 2C_0\|\nabla y_k\| \|\nabla e_k\|+ C_0 \|\nabla e_k\|^2 \le  C_0C_1\left(2 + C_1\right)\|\nabla y_k\|^2,
\end{align*}
and 
\begin{align*}
 |\alpha_k| \|\nabla G'(u_{k-2}; e_{k-1})\|  
 \le & 2C_0\|\nabla y_{k-1}\| |\alpha_k| \|\nabla e_{k-1}\| + C_0|\alpha_k|\|\nabla e_{k-1}\|^2 \\
 \le & 2C_0C_1 (\|\nabla y_k\| + \|\nabla (y_k-y_{k-1})\|)\|\nabla y_k\| + C_0 C_1^2/|\alpha_k| \|\nabla y_k\|^2 \\
 \le & 2C_0C_1\|\nabla y_k\|^2 + C_0C_1(2+C_1)/|\alpha_k|\|\nabla y_k\|^2,
 \end{align*}
thanks to triangle inequality, \eqref{eqn:alpha0} and Lemma \ref{lemma1}.
Combining the above two inequalities yields \eqref{eqn:yk2}.

Next we prove the inequality \eqref{eqn:yk1}.
For notation simplification, we denote $\psi_k := G(u_k)- G(u_{k-1}) - G'(u_{k-1}; e_k)$. Utilizing \eqref{eqn:id11}, we have identity $y_{k+1} = \tilde e_{k+1} + \alpha_k  \tilde e_k$ and then
\begin{align*}
y_{k+1}  - G'(u_{k-1}; e_k) - \alpha_k G'(u_{k-2}; e_{k-1}) =  \psi_k +\alpha_k \psi_{k-1}.
\end{align*}
From equation \eqref{eqn:errbd} and Lemma \ref{lemma1}, we have
\begin{align*}
\|\nabla \tilde e_{k+1}\|  & \le 2C_0\|\nabla y_k\| \|\nabla e_k\| + C_0\|\nabla e_k\|^2 \le C_0C_1 (2+C_1) \|\nabla y_k\|^2,
\end{align*}
 \[
\|\nabla \tilde e_k\|  \le  2C_0\|\nabla y_k\|\|\nabla e_{k-1}\| + C_0\|\nabla e_{k-1}\|^2 \le C_0C_1/|\alpha_k|(2+ C_1/|\alpha_k|) \|\nabla y_k\|^2.
\]
Combining the above inequalities with \eqref{eqn:fdg} and Lemma \ref{lemma1}, we obtain \begin{align*}
  \| \nabla (\psi_k + \alpha_k \psi_{k-1})\| 
\le & 2C_0\|\nabla e_k\| \|\tilde e_{k+1}\| + 2|\alpha_k|C_0\|\nabla e_{k-1}\|\|\nabla \tilde e_{k}\| \\
 \le&  2C_0^2C_1^2 ( 2+C_1 + 2/|\alpha_k| + C_1/|\alpha_k|^2) \|\nabla y_k \|^3  .
  \end{align*}
 \end{proof}
Now we are ready to give the superlinearly convergence result of Algorithm \ref{alg:aa1}.
\begin{thm}\label{thm:aa1}[One-step residual bound for $m=1$]
Let Assumption \ref{assump01} holds, then the residual sequence $\{y_k\}$ from Algorithm \ref{alg:aa1} satisfies 
\begin{align}
\label{eqn:supconv1}
\|\nabla y_{k+1}\| \le C_2\|\nabla y_k\|^{3/2},\quad  \forall k\ge 2,
\end{align}
where $C_2$ depends on $\nu, M, R,\|f\|_{-1},\alpha_k,\theta_k$.
\end{thm}
\begin{proof}
We begin by constructing an equation of $y_{k+1}$ using \eqref{eqn:newt}. Adding $(1-\alpha_k)$ multiple of \eqref{eqn:newt} with $j = k-1$ to $\alpha_k$ multiple of \eqref{eqn:newt} with $j= k-2$  gives
\begin{multline}
\label{eqn4}
b(u_{k-1}, u_k, v) - b(e_{k-1}, \alpha_k \tilde u_{k-1}, v) + b(y_k^\alpha, u_{k-1}, v) - \alpha_k b(y_{k-1}, e_{k-1}, v) + \nu(\nabla u_k, \nabla v) \\
= (f,v).
\end{multline}
Subtracting \eqref{eqn4} from \eqref{eqn:newt} with $j = k$ produces
\begin{multline*}
b(u_k, y_{k+1}, v) + b(y_{k+1}, u_k, v) + b(e_k , u_k, v) + \alpha_k b(e_{k-1}, \tilde u_{k-1}, v) - b(y_k^\alpha, u_{k-1}, v) \\
+ \alpha_k b(y_{k-1}, e_{k-1}, v) + \nu (\nabla y_{k+1}, \nabla v) = 0.
\end{multline*}
Setting $v = \chi: = G'(u_{k-1}; e_k) + \alpha_k G'(u_{k-2}; e_{k-1})$ gives
\begin{align*}
 & \frac{\nu}2 (\|\nabla y_{k+1}\|^2 + \|\nabla \chi\|^2 - \|\nabla (y_{k+1} - \chi )\|^2) \\
  \le & MR\|\nabla \chi\| \|\nabla (y_{k+1} - \chi )\| + \frac{MR}{2} (\|\nabla y_{k+1}\|^2 + \|\nabla \chi\|^2) + MR\|\nabla e_k\| \|\nabla \chi\|  \\
 & + M R_G |\alpha_k| \|\nabla e_{k-1}\| \|\nabla \chi\|  + \theta_k MR \| \nabla y_k\| \|\nabla \chi\| + M |\alpha_k| \|\nabla e_{k-1}\|  \|\nabla y_{k-1}\| \|\nabla \chi\|,
\end{align*}
thanks to the polarization identity, \eqref{bss} and \eqref{bbd}. 
Dropping the term with $\|\nabla \chi\|^2$ and dividing both sides by $\frac{\nu}2(1-\nu^{-1}MR)$, this reduces to
\begin{multline*} 
 \|\nabla y_{k+1}\|^2 \le (1+C_0R)  \|\nabla (y_{k+1}-\chi)\|^2  
+ 2C_0R \|\nabla \chi\| \|\nabla (y_{k+1} - \chi)\|  \\
+ \theta_k C_0R\|\nabla y_k\|\|\nabla \chi\| + 2C_0C_1(R+R_G+R_{res})\|\nabla y_k\|\|\nabla \chi\|,
\end{multline*}
thanks to Lemma \ref{lemma1}, \eqref{eqn:guubd}. Thus from Lemma \ref{lemma2}, we have \eqref{eqn:supconv1}. 
\end{proof}
\subsection{Anderson accelerated Newton's method with depth $m=2$}
In this subsection, we study the Anderson accelerated Newton's method with depth $m =2$ for solving steady NSE. The algorithm and analysis of convergence are presented here. 
\begin{alg}[Anderson accelerated Newton's iteration with depth $m = 2$ (AAN m=2)
]
\label{alg:aa2}
Algorithm of Anderson accelerated Newton's method with depth $m=2$ is stated as below:
\begin{enumerate}
\item[Step 0] Guess $u_0 \in B(0,R)$.
\item[Step 1] Compute $\tilde u_1 = G(u_0)$ and set residual $y_1 = \tilde u_1 - u_0$, update $ u_1 = \tilde u_1$.
\item[Step 2] This step consists three parts: 
\begin{enumerate}
\item[a)] Compute $\tilde u_2 = G(u_1)$ and set residual $y_2 = \tilde u_2 - u_1$.
\item[b)] Find $\alpha_2\in \mathbb{R}$ minimizing 
$$ \|\nabla ((1-\alpha_2)y_2 + \alpha_2 y_1)\|.$$
\item[c)] Update $u_2 = (1-\alpha_2)\tilde u_2 + \alpha_2 \tilde u_1$.
\end{enumerate}
\item[Step k] For $k = 3, 4,\dots$
\begin{enumerate}
\item[a)] Compute $\tilde u_k = G(u_{k-1})$ and set residual $y_k = \tilde u_k - u_{k-1}$.
\item[b)] Find $\beta_k^1,\beta_k^2\in \mathbb{R}$ minimizing 
$$ \|\nabla ((1-\beta_k^1-\beta_k^2) y_k + \beta_k^1 y_{k-1} + \beta_k^2 y_{k-2})\|.$$ %= \|\nabla (y_k - (\beta_k^1+\beta_k^2) (y_k-y_{k-1}) - \beta_k^2 (y_{k-1}- y_{k-2}))\|.$$
\item[c)] Update $u_k = (1-\beta_k^1-\beta_k^2)\tilde u_{k} + \beta_k^1 \tilde u_{k-1}+\beta_k^2\tilde u_{k-2}.$
\end{enumerate}
\end{enumerate}
\end{alg}
For smooth analysis, we will use the following notations for the rest of subsection $$e_k=u_k - u_{k-1}, \quad \tilde e_k = \tilde u_k -\tilde u_{k-1}, \quad y_k^\beta = (1-\beta_k^1-\beta_k^2)y_k +   \beta_k^1 y_{k-1} + \beta_k^2 y_{k-2}.$$
It is obvious that Algorithm \ref{alg:aa2} is back to either Algorithm \ref{alg:newt} or Algorithm \ref{alg:aa1} at step $k$ whenever $\beta_k^2 = 0.$ So it is reasonable to make the following assumption. 
\begin{assumption} 
\label{assump2}
For step $k$ ($k\ge 3$), assume $\beta_k^2 \neq 0$ and $ u_j\in B(0,R)$ with $R< \nu M^{-1}$ for all $j\le k$.
\end{assumption}
%
%\begin{dfn}
%For $k^{th}$ ($k\ge 3$) step from Algorithm \ref{alg:aa2}, let $$ \theta_k: = \|\nabla y_k^\beta \| / \|\nabla y_k\|,$$
%and is called the $k^{th}$ step Anderson gain.
%\end{dfn}
%Clearly $\theta_k \in [0,1]$ and is strictly less than 1 if $\beta_k^1,\beta_k^2 \neq 0$. When $\beta_k^2 =0$, the Algorithm \ref{alg:aa2} is back to either Algorithm \ref{alg:aa1} or  Algorithm \ref{alg:newt}. An assumption is made throughout this subsection to study the behavior of  Algorithm \ref{alg:aa2}.
%
Now let's give three lemmas that are analogue to the depth $m=1$ case.
\begin{lemma}
\label{lemma:beta}
Assume the Assumption \ref{assump2} holds, let $ \theta_k = \|\nabla y_k^\beta \| / \|\nabla y_k\|,$ then $\theta_k \in [0,1)$ and  $\beta_k^1, \beta_k^2$ from Algorithm \ref{alg:aa2} satisfy the following inequalities 
\begin{align}
\label{eqn:beta1}
|\beta_k^2 | \| \nabla (y_{k-1}-y_{k-2}))\|  & \le  (\sqrt{1-\theta_k^2} +| \beta_k^1+\beta_k^2|)\|\nabla y_k\|  + | \beta_k^1+\beta_k^2|\|\nabla  y_{k-1}\| .
%\label{eqn:beta3}
%\|\nabla (y_k - y_{k-1})\|& \le \|\nabla y_k\| + \|\nabla y_{k-1}\|.
\end{align}
\end{lemma}
\begin{proof}
%\eqref{eqn:beta3} can be obtained directly from the triangle inequality. Now we show \eqref{eqn:beta1}. 
It is obviously $\theta_k \in (0,1]$ from the minimization step.
Let $$f(\xi,\eta) = \|\nabla \left( y_k - \xi (y_k - y_{k-1}) -\eta (y_{k-1}- y_{k-2})\right) \|^2 ,$$ then $(\beta_k^1+ \beta_k^2,\beta_k^2)$ is a stationary point of  function $f$. Setting $$\frac{\partial }{\partial \xi} f = 0,  \ \frac{\partial }{\partial \eta} f = 0,$$ then $\xi = \beta_k^1+ \beta_k^2, \eta = \beta_k^2$ satisfy
\begin{align*}
\xi & = \frac{\|\nabla (y_{k-1}-y_{k-2})\|^2 (\nabla y_k, \nabla (y_k-y_{k-1})) - (\nabla (y_k-y_{k-1}), \nabla (y_{k-1}-y_{k-2}))(\nabla y_k ,\nabla ( y_{k-1}-y_{k-2}))}{\|\nabla (y_k-y_{k-1})\|^2\|\nabla (y_{k-1}-y_{k-2})\|^2 - (\nabla ( y_k-y_{k-1}),\nabla ( y_{k-1}-y_{k-2}))^2}, \\
\eta & =  \frac{\|\nabla (y_{k}-y_{k-1})\|^2 (\nabla y_k, \nabla (y_{k-1}-y_{k-2})) - (\nabla (y_k-y_{k-1}), \nabla (y_{k-1}-y_{k-2}))(\nabla y_k , \nabla (y_{k}-y_{k-1}))}{\|\nabla (y_k-y_{k-1})\|^2\|\nabla (y_{k-1}-y_{k-2})\|^2 - (\nabla ( y_k-y_{k-1}), \nabla (y_{k-1}-y_{k-2}))^2}.
\end{align*}
From the definition of $\theta_k$, we have 
\begin{align*}
 (1-\theta_k^2)\|\nabla y_k\|^2 
 & =\|\nabla (\xi(y_k-y_{k-1}) + \eta (y_{k-1}-y_{k-2}))\|^2,
\end{align*}
which implies 
\begin{align*}
|\beta_k^2| \|\nabla (y_{k-1}-y_{k-2})\| \le \sqrt{1-\theta_k^2}\|\nabla y_k\| + |\beta_k^1+\beta_k^2| \|\nabla (y_k-y_{k-1})\|,
\end{align*}
thanks to triangle inequality $\|\nabla (y_k - y_{k-1})\| \le \|\nabla y_k\| + \|\nabla y_{k-1}\|$.
\end{proof}
A few more useful identities given here besides \eqref{eqn:id12}:
\begin{align}
%\label{eqn:id4}
%\tilde u_k - u_k & = (\beta_k^1+\beta_k^2) \tilde e_k + \beta_k^2 \tilde e_{k-1},\\
\label{eqn:id21}
y_{k-1} - y_{k-2} & = \tilde e_{k-1} - e_{k-2},\\
\label{eqn:id22}
y_k^\beta & = e_k + (\beta_k^1+\beta_k^2) e_{k-1} + \beta_k^2 e_{k-2}.
%y_{k+1} & = \tilde e_{k+1} + (\beta_k^1+\beta_k^2) \tilde e_{k} + \beta_k^2 \tilde e_{k-1}
\end{align}
Note \eqref{eqn:id11} and \eqref{eqn:id13} do not hold here as the optimization step changes.
Now we bound errors $e_k$ by residuals $y_k$.
\begin{lemma}
\label{lemma3}
Let 
Assumption \ref{assump2} holds, then we have  
\begin{align}
\label{eqn:ek2bd2}
|\beta_k^2| \|\nabla e_{k-2}\|  & \le  
 C_3 (\sqrt{1-\theta_k^2} + |\beta_k^1+ \beta_k^2|) \|\nabla y_k\| +   (C_3|\beta_k^1+ \beta_k^2| + 8C_0R|\beta_k^2|)\|
  \nabla y_{k-1}\| ,\\
\label{eqn:ek1bd2}
 \|\nabla e_{k-1}\| &\le C_3 (\|\nabla y_k\|+ \|\nabla y_{k-1}\|),\\
\label{eqn:ekbd2}
\|\nabla e_k \| & \le (\sqrt{1+ C_3^2} + 2C_3  |\beta_k^1+ \beta_k^2| )\|\nabla y_k\|+ (2C_3  |\beta_k^1+ \beta_k^2| + 8C_0R|\beta_k^2|)\|\nabla y_{k-1}\|,
\end{align}
where $C_3 := \sqrt{2+2C_0R + 8C_0^2 R^2}.$
\end{lemma}
\begin{proof}
We first show inequality \eqref{eqn:ek1bd2}. Subtracting \eqref{eqn:newt} with $j = k-2$ from \eqref{eqn:newt} with $j = k-1$ yields 
  \begin{align*}
 b(\tilde e_k, u_{k-2},v) + b(\tilde u_k, e_{k-1},v) + b(u_{k-1},y_k,v) - b(u_{k-2},y_{k-1},v)  + \nu(\nabla \tilde e_k,\nabla v) = 0.
 \end{align*}
 Setting $v = e_{k-1}$ eliminates the second term and produces
 \begin{multline*}
  \frac{\nu}2(1-\nu^{-1}MR)(\|\nabla \tilde e_k\|^2 + \|\nabla e_{k-1}\|^2) \\
   \le \frac{\nu}2 \|\nabla (y_k -y_{k-1})\|^2 + MR \|\nabla y_k\| \|\nabla e_{k-1}\| + MR\|\nabla y_{k-1}\| \|\nabla e_{k-1}\| ,
 \end{multline*}
 thanks to polarization identity, $ab \le \frac{a^2 + b^2}2$ and \eqref{bbd}. Applying the Young's inequality and dropping the term with $\|\nabla \tilde e_k\|^2$ yields
 \begin{align*}
   \|\nabla e_{k-1}\|^2 
  &\le 2(1+C_0R)\|\nabla (y_k-y_{k-1})\|^2 + 8C_0^2R^2(\|\nabla y_k\|^2+ \|\nabla y_{k-1}\|^2) \\
  &\le 2(1+C_0R) (\|\nabla y_k\|  + \|\nabla y_{k-1}\|)^2 + 8C_0^2R^2(\|\nabla y_k\|+ \|\nabla y_{k-1}\|)^2 \\
 &\le 2(1+C_0R+4C_0^2R^2) (\|\nabla y_k\|+ \|\nabla y_{k-1}\|)^2.
 \end{align*}
Taking square on both sides leads to \eqref{eqn:ek1bd2}.

 Now we prove \eqref{eqn:ek2bd2}. Subtracting \eqref{eqn:newt} with $j = k-3$ from \eqref{eqn:newt} with $j = k-2$ yields 
 \begin{align*}
 b(\tilde e_{k-1}, u_{k-2},v) + b(\tilde u_{k-2}, e_{k-2}, v) + b(u_{k-3}, y_{k-1}-y_{k-2}, v) + b(e_{k-2}, y_{k-1},v) + \nu (\nabla \tilde e_{k-1}, \nabla v) = 0,
 \end{align*}
 using \eqref{eqn:id21}. 
 Setting $v = e_{k-2}$ eliminate the second term and gives
 \begin{multline*}
 \frac{\nu}2(1-\nu^{-1}MR)(\|\nabla \tilde e_{k-1}\|^2 + \|\nabla e_{k-2}\|^2) \\
 \le \frac{\nu}2 \|\nabla (y_{k-1}- y_{k-2})\|^2 +
 MR\|\nabla (y_{k-1}-y_{k-2})\|\|\nabla e_{k-2}\| + 2MR\|\nabla y_{k-1}\| \|\nabla e_{k-2}\| ,
 \end{multline*}
 thanks to the polarization identity, \eqref{bbd} and \eqref{bss}.
 Applying the Young's inequality and dropping the term with $\|\nabla \tilde e_{k-1}\|^2$ yields
 \begin{align*}
  \|\nabla e_{k-2}\|
  \le & \sqrt{C_3^2\|\nabla (y_{k-1}- y_{k-2})\|^2 + 32 C_0^2R^2 \|\nabla y_{k-1}\|^2}  \\
\le &   C_3\|\nabla (y_{k-1} - y_{k-2})\| + 8C_0R\|\nabla y_{k-1}\|.
  \end{align*}
%%  That is 
%  \begin{align*}
%  \|\nabla e_{k-2}\| \le C_3\|\nabla (y_{k-1} - y_{k-2})\| + 8C_0R\|\nabla y_{k-1}\|.
%  \end{align*}
  Multiplying both sides by $|\beta_k^2|$ and applying \eqref{eqn:beta1} produces
 \begin{align*}
   |\beta_k^2|\|\nabla e_{k-2}\|  
  \le & C_3 (\sqrt{1-\theta_k^2} + |\beta_k^1+ \beta_k^2|) \|\nabla y_k\| +   (C_3|\beta_k^1+ \beta_k^2| + 8C_0R|\beta_k^2|)\|
  \nabla y_{k-1}\|.
  \end{align*}
%  which leads to \eqref{eqn:ek2bd2}.

 Lastly, we show \eqref{eqn:ekbd2}. From \eqref{eqn:id22}, we have
 \begin{align*}
 \|\nabla e_k\| & \le \theta_k\|\nabla y_k\| + |\beta_k^1+\beta_k^2|\|\nabla e_{k-1}\| + |\beta_k^2|\|\nabla e_{k-2}\| \\
 & \le (\theta_k + 2C_3  |\beta_k^1+ \beta_k^2| + C_3\sqrt{1-\theta_k^2}) \|\nabla y_k\| + (2C_3  |\beta_k^1+ \beta_k^2| + 8C_0R|\beta_k^2|)\|\nabla y_{k-1}\| \\
 &\le (\sqrt{1+ C_3^2} + 2C_3  |\beta_k^1+ \beta_k^2| )\|\nabla y_k\|+ (2C_3  |\beta_k^1+ \beta_k^2| + 8C_0R|\beta_k^2|)\|\nabla y_{k-1}\|,
 \end{align*}
due to $\max\limits_{\theta\in[0,1)}\{ \theta + C_3 \sqrt{1-\theta^2}\} = \sqrt{1+C_3^2}.$
%it reduces to \eqref{eqn:ekbd2} due to the definition of $C_4.$ We finish the proof.
\end{proof}
\begin{lemma}
\label{lemma4}
Let Assumption \ref{assump2} holds, then we have
\begin{align}
\label{eqn:yk3}
\|\nabla (G'(u_{k-1};e_k) + (\beta_k^1+\beta_k^2)G'(u_{k-2};e_{k-1})+ \beta_k^2 G'(u_{k-3}; e_{k-2}))\| \le  \mathcal{O}((\|\nabla y_k\| + \|\nabla y_{k-1}\|)^2),
\end{align}
and 
\begin{multline}
\label{eqn:yk4}
\| \nabla ( y_{k+1} - G'(u_{k-1}; e_k) - (\beta_k^1+\beta_k^2) G'(u_{k-2};e_{k-1}) -\beta_k^2 G'(u_{k-3};  e_{k-2}) ) \| 
\\ \le  \mathcal{O}((\|\nabla y_k\| + \|\nabla y_{k-1}\|)^3).
\end{multline}
\end{lemma}
\begin{proof}
From \eqref{eqn:dgbd} and Lemma \ref{lemma3}, we have
\begin{align*}
 \|\nabla G'(u_{k-1};e_k)\|  
\le  & 2C_0\|\nabla y_k\|\|\nabla e_k\| + C_0\|\nabla e_k\|^2 \\
\le & \mathcal{O}(\|\nabla y_k\|(\|\nabla y_k\| + \|\nabla y_{k-1}\|)) +\mathcal{O}( (\|\nabla y_k\| + \|\nabla y_{k-1}\|)^2), \\
  |\beta_k^1+\beta_k^2| \|\nabla G'(u_{k-2};e_{k-1})\|  
  \le & 2C_0\|\nabla y_{k-1}\|  |\beta_k^1+\beta_k^2|  \|\nabla e_{k-1}\|  + C_0 |\beta_k^1+\beta_k^2|  \|\nabla e_{k-1}\|^2 \\
\le &\mathcal{O}(\|\nabla y_{k-1}\| (\|\nabla y_k\| + \|\nabla y_{k-1}\|) ) + \mathcal{O}((\|\nabla y_k\| + \|\nabla y_{k-1}\|)^2),
\end{align*}
and
\begin{align*}
|\beta_k^2|\|\nabla G'(u_{k-3}; e_{k-2})\| &\le 2C_0\|\nabla y_{k-1}\| |\beta_k^2| \|\nabla e_{k-2}\| + C_0|\beta_k^2|\|\nabla e_{k-2}\|^2 \\
& \le \mathcal{O}( \|\nabla y_{k-1}\| (\|\nabla y_k\| + \|\nabla y_{k-1}\|)) + \mathcal{O}((\|\nabla y_k\| + \|\nabla y_{k-1}\|)^2) .
\end{align*}
Combining the above three inequalities, we have 
\begin{align*}
&\|\nabla (G'(u_{k-1};e_k) + (\beta_k^1+\beta_k^2)G'(u_{k-2};e_{k-1})+ \beta_k^2 G'(u_{k-3}; e_{k-2}))\| \\
\le & \|\nabla G'(u_{k-1};e_k)\| +  |\beta_k^1+\beta_k^2| \|\nabla G'(u_{k-2};e_{k-1})\|  + |\beta_k^2| \|\nabla G'(u_{k-3}; e_{k-2})\|  \\
\le & \mathcal{O}((\|\nabla y_k\| + \|\nabla y_{k-1}\|)^2).
\end{align*}
For notation simplification, denote $\psi_k = G(u_k)-G(u_{k-1}) - G'(u_{k-1};e_k)$. Utilizing identity 
$y_{k+1} = \tilde e_{k+1} + (\beta_k^1+\beta_k^2) \tilde e_k + \beta_k^2 \tilde e_{k-1}$ yields
\begin{align*}
 &y_{k+1} - G'(u_{k-1}; e_k) -( \beta_k^1+\beta_k^2) G'(u_{k-2};e_{k-1}) -\beta_k^2 G'(u_{k-3};  e_{k-2})  \\
 = & \psi_k + (\beta_k^1+\beta_k^2) \psi_{k-1} + \beta_k^2 \psi_{k-2}.
\end{align*}
From equation \eqref{eqn:errbd} and Lemma \ref{lemma3}, we have the following three inequalities
\begin{align*}
\|\nabla \tilde e_{k+1}\| & \le 2C_0\|\nabla e_k\| \|\nabla y_k\| + C_0\|\nabla e_k\|^2  \le  \mathcal{O}((\|\nabla y_k\| + \|\nabla y_{k-1}\|)^2),\\
\|\nabla \tilde e_k\| & \le 2C_0\|\nabla e_{k-1}\| \|\|\nabla y_k\| + C_0\|\nabla e_{k-1}\|^2  \le  \mathcal{O}((\|\nabla y_k\| + \|\nabla y_{k-1}\|)^2),\\
\|\nabla \tilde e_{k-1}\| & \le 2C_0\|\nabla e_{k-2}\| \|\nabla y_{k-1}\| + C_0\|\nabla e_{k-2}\|^2  \le \mathcal{O}((\|\nabla y_k\| + \|\nabla y_{k-1}\|)^2).
\end{align*}
Utilizing \eqref{eqn:fdg} gives
\begin{align*}
& \|\nabla (y_{k+1} - G'(u_{k-1}; e_k) -( \beta_k^1+\beta_k^2) G'(u_{k-2};e_{k-1}) -\beta_k^2 G'(u_{k-3};  e_{k-2}) )\|\\ 
\le & \|\nabla  \psi_k\| +  |\beta_k^1+\beta_k^2|\|\nabla  \psi_{k-1} \|+ |\beta_k^2|\|\nabla  \psi_{k-2}\| \\
\le & 2C_0\|\nabla e_k\| \|\nabla \tilde e_{k+1}\| + 2|\beta_k^1+\beta_k^2| C_0\|\nabla e_{k-1}\| \|\nabla \tilde e_k\| + 2C_0|\beta_k^2|\|\nabla e_{k-2}\| \|\nabla \tilde e_{k-1}\| \\
\le & \mathcal{O}( ( \|\nabla y_k\| +\|\nabla y_{k-1}\|)^3).
\end{align*}
\end{proof}
Now we manifest that Anderson accelerated Newton's method with depth 2 also converges superlinearly.
\begin{thm}[One-step residual bound for $m=2$] Let Assumption \ref{assump2} holds, then the residual sequence $\{y_k\}$ from Algorithm \ref{alg:aa2} satisfies
$$ \|\nabla y_{k+1}\| \le  \mathcal{O}((\|\nabla y_k\| + \|\nabla y_{k-1}\|)^{3/2}),$$
where the bound depends on parameters $\nu, |\Omega|, R, f , \beta_k^1, \beta_k^2,\theta_k$.
\label{thm:aa2}
\end{thm}
\begin{proof}
Start with constructing $u_k$. Adding $1-\beta_k^1-\beta_k^2$ multiple of \eqref{eqn:newt} with $j = k-1$, $\beta_k^1$ multiple of \eqref{eqn:newt} with $j = k-2$ to $\beta_k^2$ multiple of \eqref{eqn:newt} with $j =k-3$ gives
%\begin{multline}
%b(u_{k-3},u_k,v) + b(e_{k-2},\beta_k^1\tilde u_{k-1}, v) + b(e_{k-1}+e_{k-2}, (1-\beta_k^1-\beta_k^2)\tilde u_k,v) 
%+b(y_k^\beta, u_{k-3},v) \\+ b(\beta_k^1 y_{k-1}, e_{k-2},v) + (1-\beta_k^1-\beta_k^2)b(y_k,e_{k-1}+e_{k-2}) 
%+\nu(\nabla u_k,\nabla v) = (f,v).
%\end{multline}
%OR 
%\textcolor{blue}{
%\begin{multline*}
%\label{eqn6}
%b(u_{k-1},u_k,v) -\beta_k^1 b(e_{k-1},\tilde u_{k-1},v) -\beta_k^2b(e_{k-1}+e_{k-2}, \tilde u_{k-2},v)
%+b(y_k^\beta, u_{k-3},v) \\+ \beta_k^1b( y_{k-1}, e_{k-2},v) + (1-\beta_k^1-\beta_k^2)b(y_k,e_{k-1}+e_{k-2},v) 
%+\nu(\nabla u_k,\nabla v) = (f,v).
%\end{multline*}
%OR
%\begin{multline*}
%b(u_{k-3},u_k,v) + (1-\beta_k^1-\beta_k^2)b(e_{k-1}+e_{k-2}, \tilde u_k, v) + \beta_k^1b(e_{k-1}, \tilde u_{k-1},v) \\+ \beta_k^1 b( y_{k-1}, e_{k-2},v) + (1-\beta_k^1-\beta_k^2)b(y_k,e_{k-1}+e_{k-2},v) 
%+\nu(\nabla u_k,\nabla v) = (f,v).
%\end{multline*}
%OR
%\begin{multline*}
%b(u_{k-1},u_k,v) -\beta_k^1 b(e_{k-1},\tilde u_{k-1},v) -\beta_k^2b(e_{k-1}+e_{k-2}, \tilde u_{k-2},v)
%+b(y_k^\beta, u_{k-1},v) \\
%- \beta_k^1b( y_{k-1}, e_{k-1},v) -\beta_k^2b(y_{k-2},e_{k-1}+e_{k-2},v) 
%+\nu(\nabla u_k,\nabla v) = (f,v).
%\end{multline*}
%}
\begin{multline*}
b(y_k , u_{k-1},v) -(\beta_k^1+\beta_k^2)b(y_k-y_{k-1},u_{k-2},v) - (\beta_k^1+\beta_k^2)b(y_k,e_{k-1},v)  \\
-\beta_k^2b(y_{k-1}-y_{k-2},u_{k-3},v)  -\beta_k^2b(y_{k-1},e_{k-2},v) 
+  b(u_{k-1},u_k,v) -\beta_k^1b(e_{k-1},\tilde u_{k-1}, v) \\
 - \beta_k^2 b(e_{k-1}+e_{k-2}, \tilde u_{k-2},v) 
+ \nu(\nabla u_k,\nabla v) = (f,v).
\end{multline*}
Subtracting it from \eqref{eqn:newt} with $j = k$ yields
\begin{multline*}
b(y_{k+1}, u_k,v) + b(e_k,u_k,v) + b(u_k, y_{k+1},v) - b(y_k , u_{k-1},v) + (\beta_k^1+\beta_k^2)b(y_k-y_{k-1},u_{k-2},v) \\
 + (\beta_k^1+\beta_k^2)b(y_k,e_{k-1},v) + \beta_k^2b(y_{k-1}-y_{k-2},u_{k-3},v) 
+\beta_k^2b(y_{k-1},e_{k-2},v) \\
 + b(e_{k-1},  \beta_k^1\tilde u_{k-1} + \beta_k^2\tilde u_{k-2},v) + \beta_k^2 b(e_{k-2}, \tilde u_{k-2}, v)
+ \nu (\nabla y_{k+1}, \nabla v) = 0.
\end{multline*}
Setting $v  = \chi := G'(u_{k-1}; e_k) + (\beta_k^1+\beta_k^2) G'(u_{k-2};e_{k-1}) +\beta_k^2 G'(u_{k-3};e_{k-2})$ produces
\begin{align*} 
& \frac{\nu}2 (\|\nabla y_{k+1}\|^2 + \|\nabla \chi\|^2 - \|\nabla (y_{k+1} -\chi )\|^2)  \\
\le &  
MR\|\nabla y_{k+1}\| \|\nabla \chi\| + MR\|\nabla e_k\| \|\nabla \chi\| + MR\|\nabla (y_{k+1} - \chi)\| \|\nabla \chi\|  
+MR\|\nabla y_k\| \|\nabla \chi\|  \\
&
+ |\beta_k^1+\beta_k^2| MR\|\nabla (y_k - y_{k-1})\| \|\nabla \chi\| 
+ |\beta_k^1+\beta_k^2| M\|\nabla y_k \| \|\nabla e_{k-1}\| \|\nabla \chi\|   \\
& + |\beta_k^2|MR  \|\nabla (y_{k-1}-y_{k-2})\| \|\nabla  \chi\| + |\beta_k^2| M\|\nabla y_{k-1}\| \|\nabla e_{k-2}\|\|\nabla \chi\|  \\
&  + M \|\nabla (u_k - (1-\beta_k^1-\beta_k^2)\tilde u_k)\| \|\nabla e_{k-1}\| \|\nabla \chi\| + |\beta_k^2|MR_G\|\nabla e_{k-2}\| \|\nabla \chi\|,
\end{align*}
thanks to polarization identity, \eqref{bss}, \eqref{bbd}. 
Applying $ \|\nabla (u_k - (1-\beta_k^1-\beta_k^2)\tilde u_k)\|  \le R+ (1+|\beta_k^1+\beta_k^2|)R_G $, Young's inequality,  and dropping $\|\nabla \chi\|^2$ term, we obtain
\begin{align*}
  \|\nabla y_{k+1}\|^2  
 \le &  (1+C_0R)  \|\nabla (y_{k+1} -\chi )\|^2 +  2C_0R\|\nabla (y_{k+1} - \chi)\| \|\nabla \chi\| \\
 &+4C_0R (1+  |\beta_k^1+\beta_k^2|) (\|\nabla y_k\| +\|\nabla y_{k-1}\| )\|\nabla \chi\| \\
 &+  2C_0R\|\nabla e_k\| \|\nabla \chi\|  + 2C_0 \|\nabla y_k\| |\beta_k^1+\beta_k^2|\|\nabla e_{k-1}\|\|\nabla \chi\| \\
 & + 2C_0(R + R_G (1+  |\beta_k^1+\beta_k^2|))\|\nabla e_{k-1}\| \|\nabla \chi\| \\
 & + 2C_0\|\nabla y_{k-1}\| |\beta_k^2| \|\nabla e_{k-2}\|\|\nabla \chi\|   + 2C_0R_G |\beta_k^2| \|\nabla e_{k-2}\| \|\nabla \chi\|
 \\ 
\le & \mathcal{O}((\|\nabla y_k\| + \|\nabla y_{k-1}\|)^3).
 \end{align*}
 thanks to Lemma \ref{lemma:beta}, Lemma \ref{lemma3} and Lemma \ref{lemma4}.
\end{proof}
%\begin{cor}
%The domain of convergence for Algorithm \ref{alg:aa2} is 
%$$ D_2 = \{ v\in B(0,R)\subset V_h \mid \|\nabla (G(v) -v)\| \le d_2\}$$
%where $d_2 := .$
%\end{cor}
\subsection{Anderson accelerated Newton's method with depth $m$}
This subsection states the algorithm and one-step convergence result of the general Anderson acceleration applied to Newton's method for solving steady Navier-Stokes equations. The analysis would be similar to the previous two subsections but involving more complicate constants, and so here we omit it.
\begin{alg}[Anderson accelerated Newton's iteration with depth $m$ for NSE]
\label{alg:aam}
Algorithm of Anderson accelerated Newton's method with depth $m$ is stated as below:
\begin{enumerate}
\item[Step 0] Guess $u_0 \in B(0,R)$.
\item[Step 1] Compute $\tilde u_1 = G(u_0)$ and set residual $y_1 = \tilde u_1 - u_0$, update $u_1 = \tilde u_1$.
\item[Step k] For $k = 2,3,\dots$, set $m_k = \min\{k-1,m\}$
\begin{enumerate}
\item[a)] Compute $\tilde u_k = G(u_{k-1})$ and set $y_k = \tilde u_k - u_{k-1}$.
\item[b)] Find $\{\gamma_k^i\}_{i=1}^{m_k}\subset \mathbb{R}$ minimizing 
$$ \left\| \nabla \left( \left(1-\sum\limits_{i=1}^{m_k} \gamma_k^i\right) y_k +  \sum\limits_{i=1}^{m_k}  \gamma_k^i y_{k-i}  \right) \right\|.$$
\item[c)] Update $u_k = \left(1-\sum\limits_{i=1}^{m_k} \gamma_k^i\right) \tilde u_k +  \sum\limits_{i=1}^{m_k}  \gamma_k^i \tilde u_{k-i}$.
\end{enumerate}
\end{enumerate}
\end{alg}
Clearly, Algorithm \ref{alg:aam} is back to AAN with small Anderson depth or Newton's method if $\gamma_k^m =0$ for any $k\ge m+1$. So we will assume $\gamma_k^m \neq 0$.

\begin{thm}[One-step residual bound]
\label{thm:aa3}
 Assume for any step $k>m$ with $\gamma_k^m \neq 0$ and $u_j \in B(0,R)$ with $R < \nu M^{-1}$ for all $j\le k$, then 
$$ \|\nabla y_{k+1}\| \le \mathcal{O}\left( \left(\sum\limits_{i =1}^{m}\|\nabla y_{k-i+1}\| \right)^{3/2} \right),$$
where the bound depends on parameters $\nu, |\Omega|, R, f , \gamma_k^i,\theta_k$.
\end{thm}
This theorem tells the general Anderson accelerated Newton's method for solving NSE converges superlinearly if the initial guess is good enough, and large depth algorithm converges slower than small depth algorithm. 
\section{Numerical tests}
In this section, we test two benchmark problems to verify the superlinearly convergence of Anderson accelerated Newton's method for solving steady Navier-Stokes equation. %From test results, we find that Anderson acceleration  enlarges the domain of convergence, but lowers the convergence order for Newton's method.
\subsection{2D cavity problem}
The 2D driven cavity uses a domain $\Omega = [0, 1]^2$, with no slip boundary conditions on the sides and
bottom, and a `moving lid' on the top which is implemented by enforcing the Dirichlet boundary condition
$u(x, 1) = (1, 0)^T$, no forcing ($f = 0$). We discretize with $(P_2, P_1)$ Taylor-Hood elements on a $1/64$ mesh that provides $44,365$ total degrees of freedom, and for the initial guess we used the Picard solution after 3 iterations on the same mesh with the same finite element setting. Newton's method and Anderson accelerated Newton's methods with several depths are tested with tolerance $1e-13$.

Streamline plots at $Re = 2500,\ 5000$ obtained by Anderson accelerated Newton's method with depth $m =5$ are given in Figure \ref{fig0}, which are in well agreement with literature \cite{BS06}.  In fact, similar plots can be observed for depth $m = 1,2,10$ too, and therefore are omitted here.  
Convergence plot at various Reynolds numbers are shown in Figure \ref{fig1}. 
For problem with $Re = 2500$, Anderson acceleration slows down the convergence speed to Newton's method, but is still faster than linearly convergent solvers, say Picard iteration, or Anderson accelerated Picard method \cite{PRX19}. For $Re=5000$, we observe that Anderson acceleration enlarges the domain of convergence for Newton's method, which implies Anderson acceleration Newton's method is worth to use in practice, especially for problems with small domain of convergence.  For both cases $Re =2500, 5000$, we find that large depth Anderson accelerated Newton's method converges slower than small depth algorithm, that matches our analytical results. 
In addition,  Table \ref{NAArate2d} shows that Anderson accelerated Newton's method $m=1$ converges superlinearly with order close to 1.5, whereas large depth decelerates the convergence speed for Anderson accelerated Newton's method, which matches our theoretical results from Theorem \ref{thm:aa1}, Theorem \ref{thm:aa2} and Theorem \ref{thm:aa3}.
%A list of convergence order of Anderson accelerated Newton's method with depth $m =1$ is given in Table \ref{t1}, which indicates the convergence order is about $3/2$, that is well agreement with Theorem \ref{thm:aa1}.
%\begin{table}[h]
%\center
%\begin{tabular}{c|c|c|c|c|c}
%$m$ & 0 & 1& 2& 5&10 \\ \hline
% median order & & & & &
%\end{tabular}
%\caption{\label{t1}Convergence order of Newton-Anderson method $m=0, 1$, (median value of convergence rate)}
%\end{table}
%
\begin{figure}[h!]
\center 
\includegraphics[trim=100 40 20 30,clip,width=.32\textwidth]{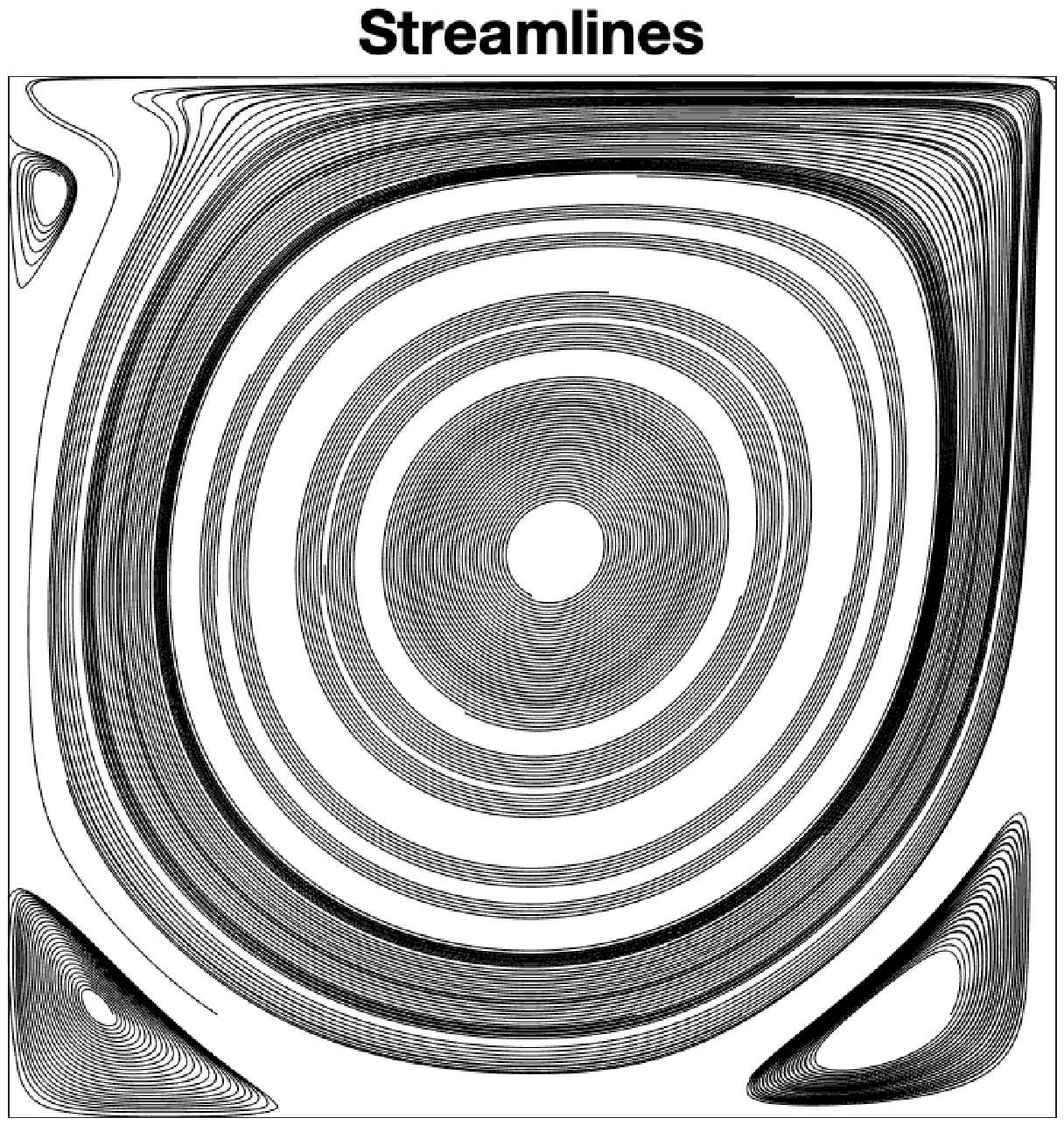}
\includegraphics[trim=100 40 20 30,clip,width=.32\textwidth]{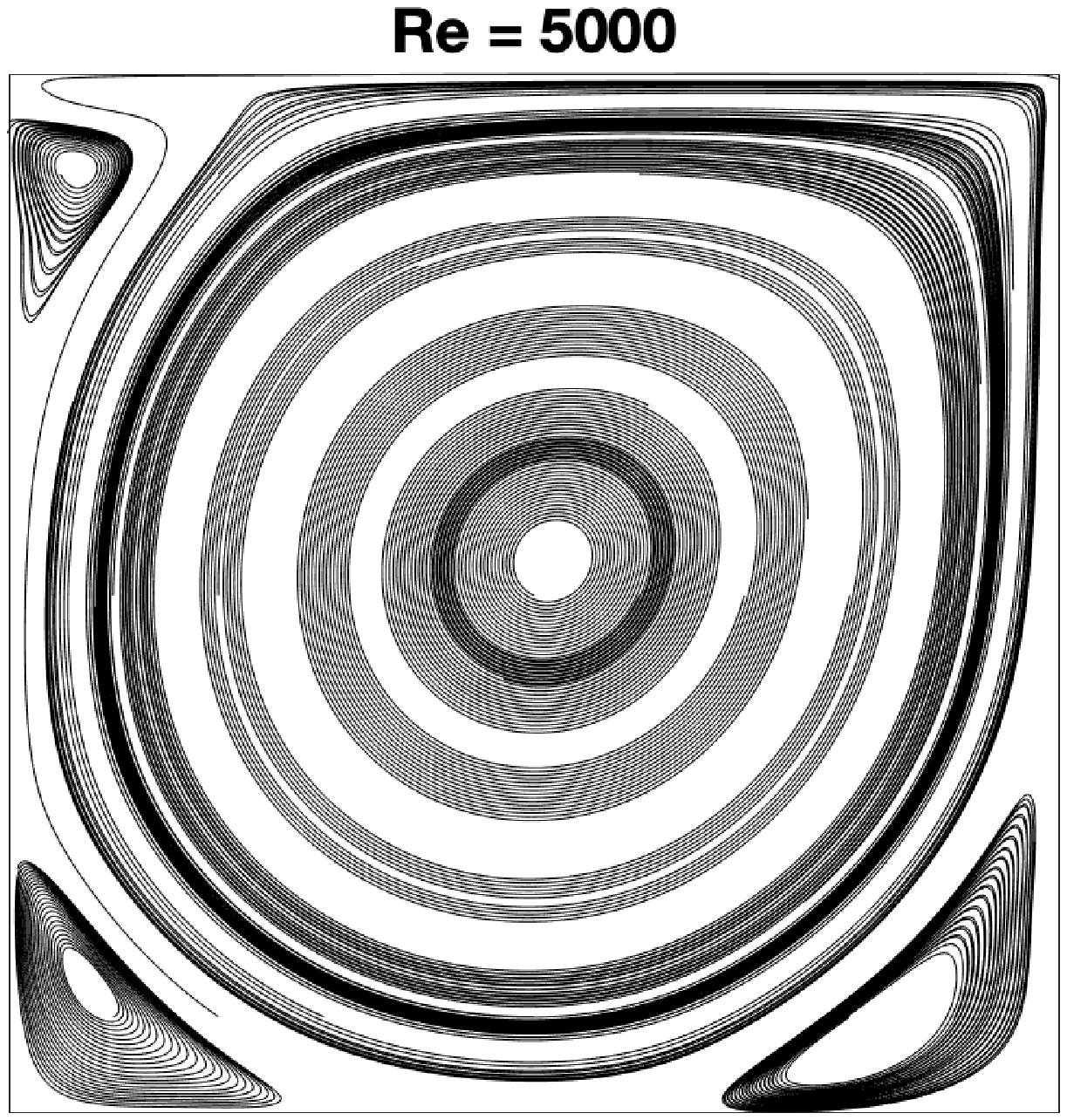}
\caption{\label{fig0} Shown above are streamline plots of the solutions from Anderson accelerated Newton solvers at Reynolds number $Re = 2500$ (left), 5000 (right). }
\end{figure}
\begin{figure}[h!]
\center
\includegraphics[trim=0 40 0 30,clip,width=.39\textwidth]{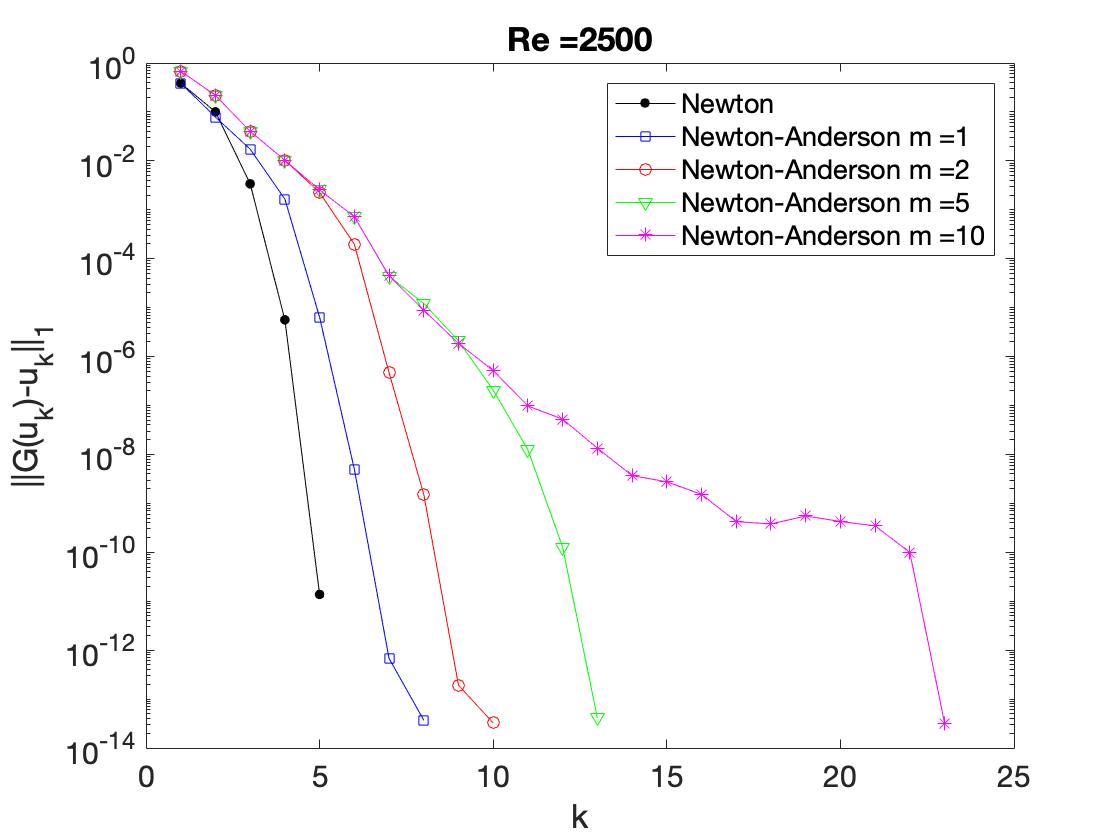}
\includegraphics[trim=0 40 0 30,clip,width=.39\textwidth]{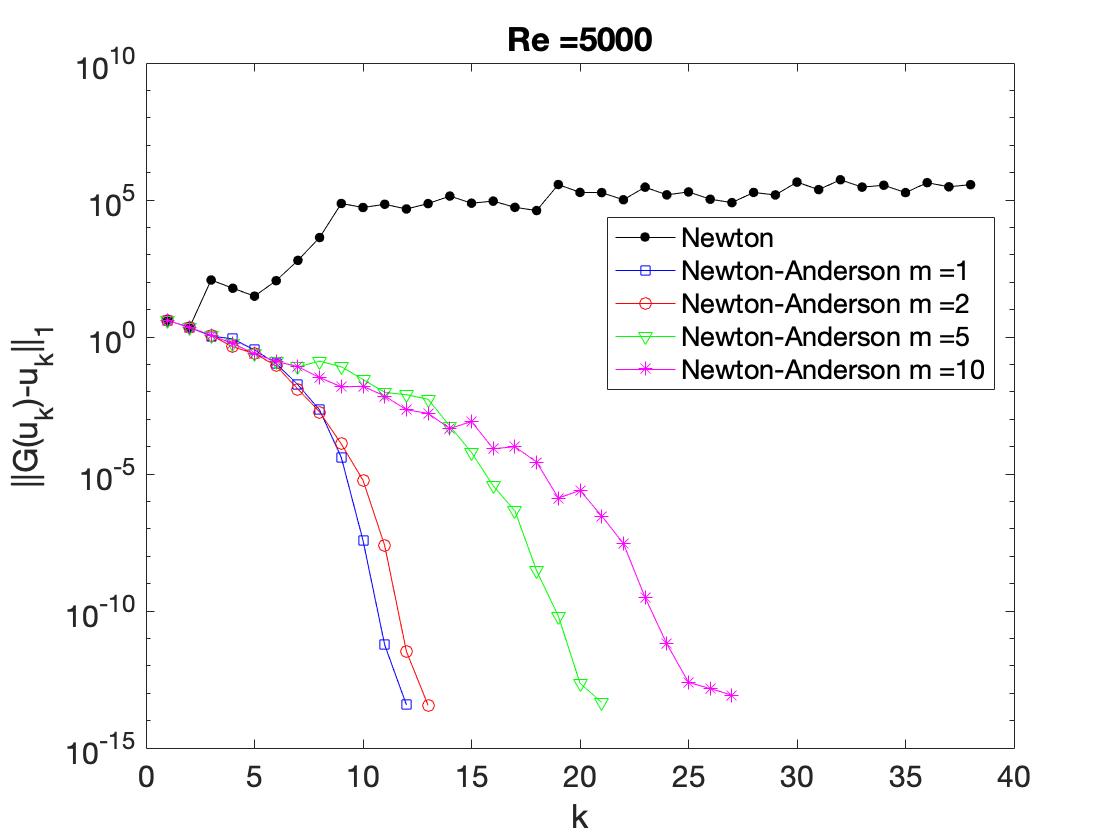}
\caption{\label{fig1} Shown above are convergence plots of Anderson accelerated Newton's method with various depth $m$  for the 2D cavity at Reynolds number $Re = 2500$ (left), 5000 (right). }%Bottom is the boxplot of convergence order to Anderson accelerated Newton's method with depth $m=1$ at various Reynolds numbers.}
\end{figure}
\begin{table}[h!]
\begin{center}
\begin{tabular}{|c|ccccc|}
\hline
conv. order & Newton & AAN $m=1$ & AAN $m=2$ & AAN $m=5$ & AAN $m=10$ \\ \hline
$Re = 2500$ &2.0192&   1.2662 &   1.2519&   1.3072&   1.3936\\
$Re = 5000$ &  Fail & 1.2519 &  1.3936 &1.0203&  0.82909\\ \hline
\end{tabular}
\end{center}
\caption{\label{NAArate2d} This table summarizes the median of convergence order for Newton's method, Anderson accelerated Newton's method (AAN) for different $Re = 2500,5000$ in the 2D cavity problem.}
\end{table}

\subsection{3D cavity problem}
In this subsection, we test Anderson accelerated Newton's method on a 3D lid driven cavity problem with Reynolds number $Re =400, \ 1000$. 
We use a domain $\Omega = [0, 1]^3$, with no slip boundary conditions on all walls, and a unite `moving lid' $u = (1,0,0)^T$ on the top, no forcing ($f = 0$). We discretize with $(P_3, P_2^{disc})$ Scott-Vogelius elements on a barycenter refined uniform mesh that provides $206,874$ total degrees of freedom, and  use zero interior initial guess but satisfying the boundary conditions. Newton's method and Anderson accelerated Newton's methods with several depths are tested with tolerance $1e-6$.

We solve the saddle point linear systems that arise at each iteration via method in \cite{BO06}. Decompose the coefficient matrix via a LU block factorization
\begin{align*}
\begin{pmatrix}
A_k & B\\ B^T & 0 
\end{pmatrix} \begin{pmatrix}
U_k \\ P_k
\end{pmatrix}
 = \begin{pmatrix}
 A_k & 0 \\ B^T & - B^T A_k^{-1}B
 \end{pmatrix} \begin{pmatrix}
 I & A_k^{-1}B^T \\ 0&I
 \end{pmatrix}  \begin{pmatrix}
 U_k \\ P_k
 \end{pmatrix}= \begin{pmatrix}
 F \\ G
 \end{pmatrix}.
\end{align*}
This leads to two solves of a smaller size linear system with coefficient matrix $A_k$ and one solve of a linear system with coefficient matrix to be the Schur complement $B^TA_k^{-1}B$. Direct solver and BICGSTAB with tolerance $1e-10$ and preconditioner pressure mass matrix were used to solve these two types linear systems respectively.

 Plots of centerline $x-$velocity and centerplane slices obtained from Anderson accelerated Newton's method
with $m =1$ are given in Figure \ref{fig2}, which matches well with \cite{WB02}. In fact, similar plots can be observed from Anderson accelerated Newton's method with other depth provided the method converges and are omitted.
\begin{figure}[h!]
\center
\includegraphics[trim=0 0 10 10,clip,width=.29\textwidth]{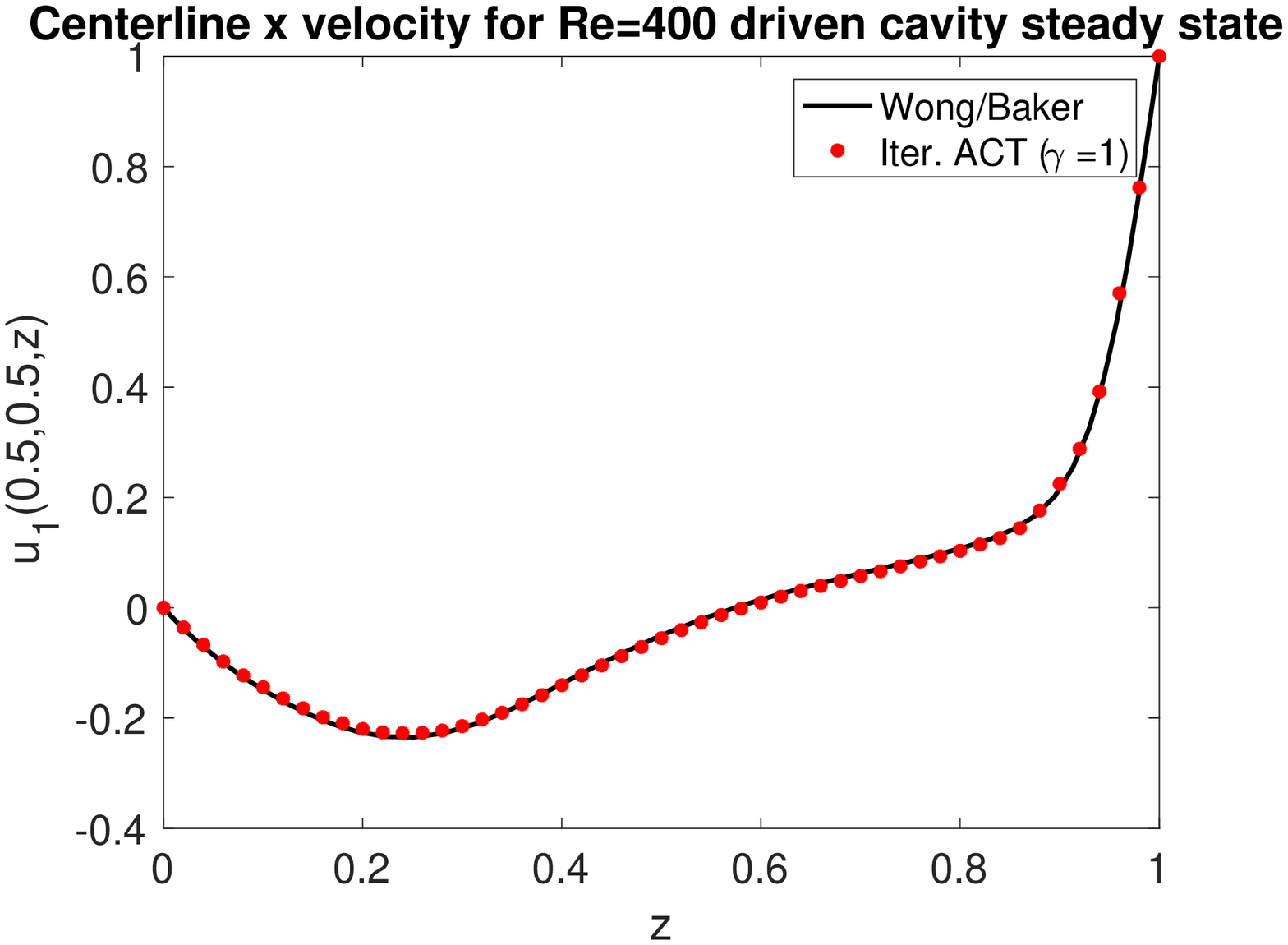}
\includegraphics[trim=50 130 40 120,clip,width=.69\textwidth]{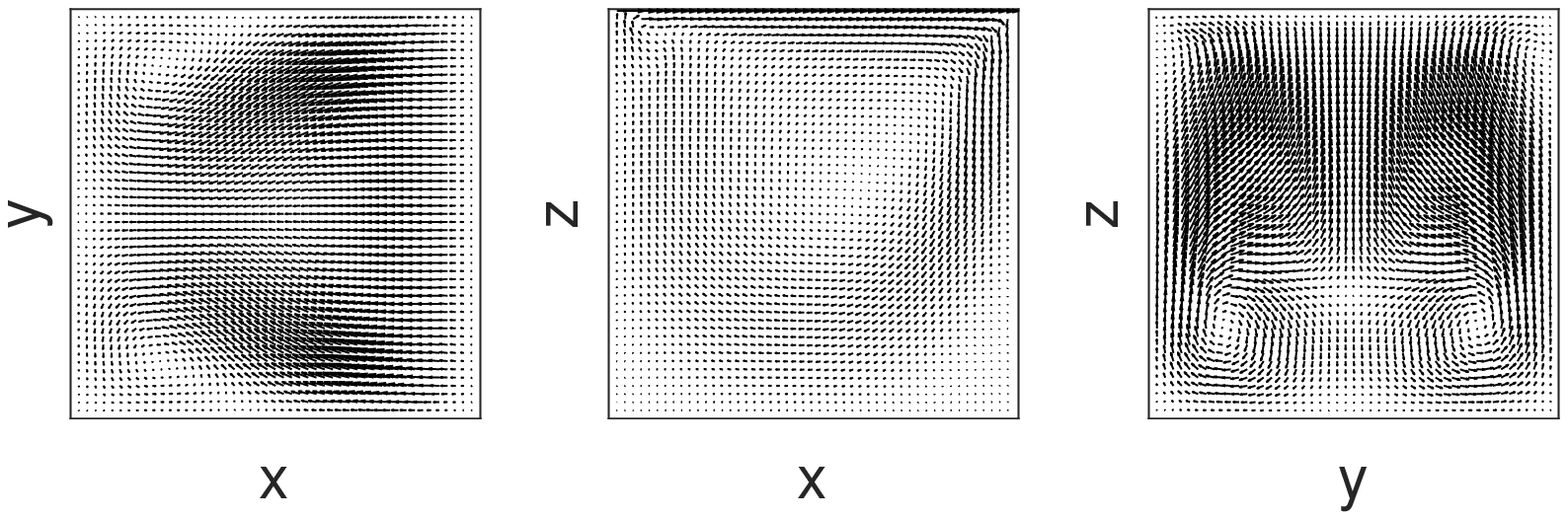}\\
\includegraphics[trim=0 0 10 10,clip,width=.29\textwidth]{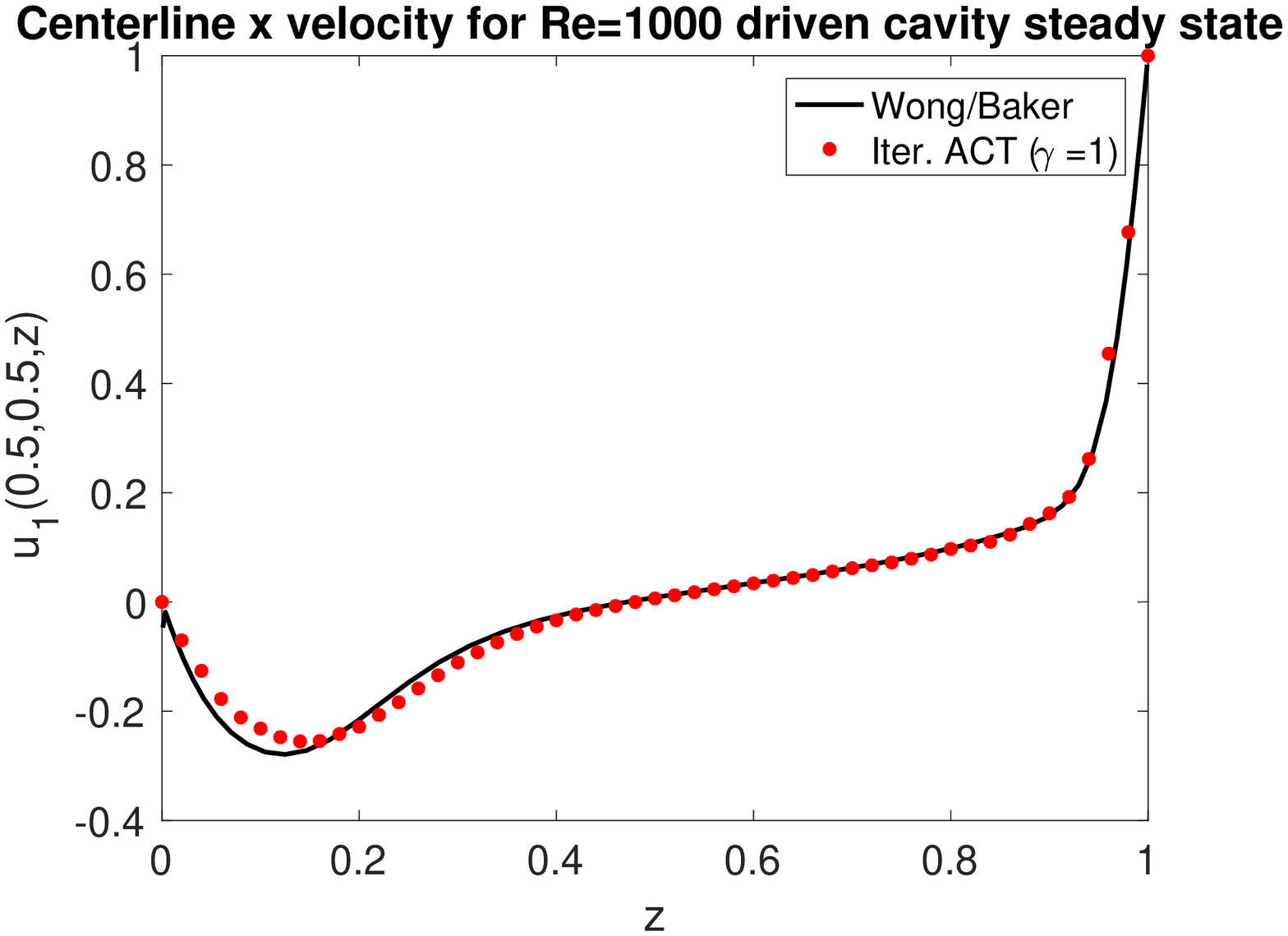}
\includegraphics[trim=50 130 40 120,clip,width=.69\textwidth]{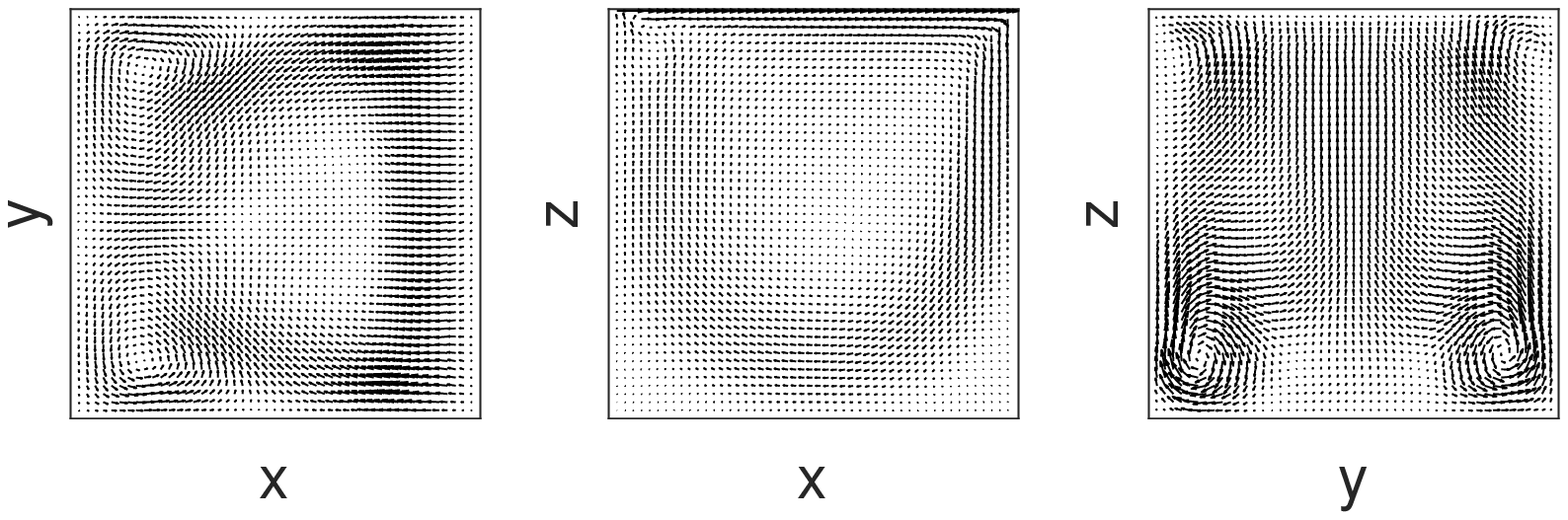}
\caption{\label{fig2} Shown above are the centerline $x-$velocity and centerplane slices for 3D cavity problem with $Re =400$ (top), $1000$ (bottom), obtained from Anderson accelerated Newton's method with depth $m=1$. }
\end{figure}
Convergence plots for 3D cavity problem with $Re = 400, 1000$ are given in Figure \ref{fig4}. We observe superlinear convergence for Anderson accelerated Newton's method when $Re = 400$, which matches our analytical results well. However, when $Re =1000$ we find that Anderson accelerated Newton's method with depth $m =2,5,10$ fail due to smaller domain of convergence while $m=1$ converges superlinearly.  A safeguard strategy would be using Picard iteration or Anderson accelerated Picard iteration first and then switch to Newton's method with or without Anderson acceleration when residual is small enough.
%Newton's method converges for both $Re = 400, 1000.$ Superlinearly convergence is observed for all depth Anderson acceleration, and large depth converges slower than small depth, which confirms our analytical results. 
\begin{figure}[h!]
\center
\includegraphics[trim=0 0 10 30,clip,width=.39\textwidth]{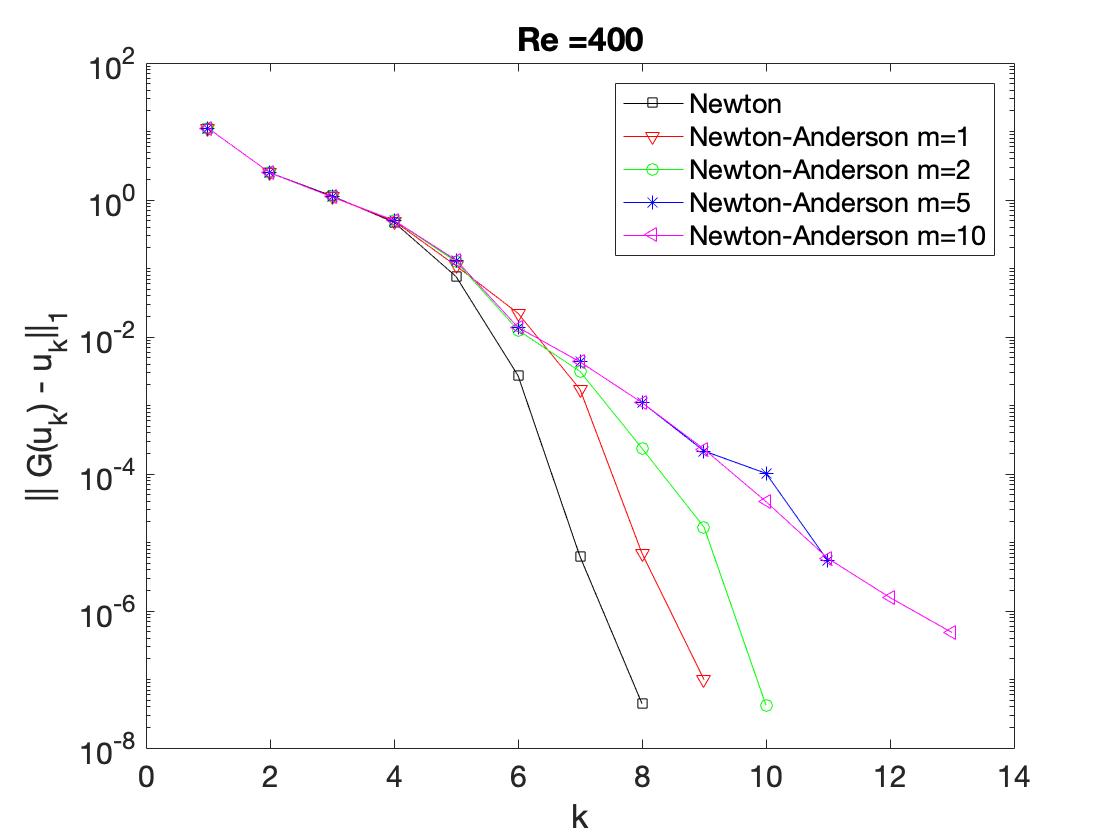}%3Dcavity_convAAN_Re400}
\includegraphics[trim=0 0 10 30,clip,width=.39\textwidth]{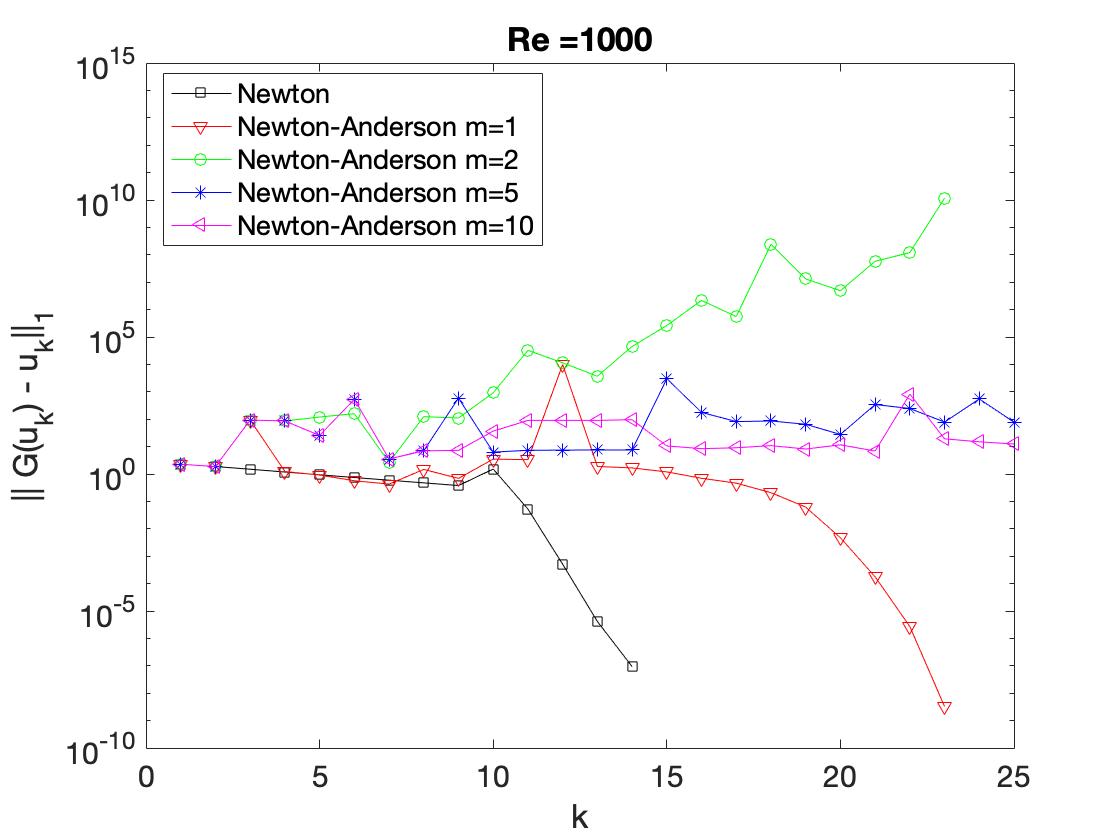}%3Dcavity_convAAN_Re1000}
\caption{\label{fig4} Shown above are convergence plot of Anderson accelerated Newton's method at $Re = 400$ (left), 1000 (right).} %Bottom is the boxplot of convergence order to Anderson accelerated Newton's method with depth $m=1$ at various Reynolds numbers.}
\end{figure}

\section{Conclusions}
In this paper, we have studied the performance of Anderson acceleration to Newton's method for solving steady Navier-Stokes equations. We find that Anderson accelerated Newton's method  with a good initial guess converges superlinearly%with order $3/2$
, which is slower than the usual Newton's method. Moreover, Anderson acceleration with large depth decelerates the convergence speed comparing to the one with small depth. The numerical tests confirm our analytical results. In addition, we observe that Anderson acceleration sometimes enlarges the domain of convergence from the 2D cavity experiment with $Re = 5000$, but sometimes narrow the domain of convergence from the 3D cavity experiment with $Re = 1000$, this phenomenon is unexplained and will be studied in the near future.

%% -- --------------------------------------------------------------------
%% -- --------------------------------------------------------------------
\bibliographystyle{plain}
\bibliography{graddiv}

\begin{thebibliography}{10}

\bibitem{Anderson65}
D.~G. Anderson.
\newblock Iterative procedures for nonlinear integral equations.
\newblock {\em J. Assoc. Comput. Mach.}, 12(4):547--560, 1965.

\bibitem{ABF84}
D.~Arnold, F.~Brezzi, and M.~Fortin.
\newblock {A stable finite element for the Stokes equations}.
\newblock {\em Calcolo}, 21(4):337--344, 1984.

\bibitem{arnold:qin:scott:vogelius:2D}
D.~Arnold and J.~Qin.
\newblock Quadratic velocity/linear pressure {S}tokes elements.
\newblock In R.~Vichnevetsky, D.~Knight, and G.~Richter, editors, {\em Advances
  in Computer Methods for Partial Differential Equations VII}, pages 28--34.
  IMACS, 1992.

\bibitem{BO06}
M.~Benzi and M.~Olshanskii.
\newblock An augmented {L}agrangian-based approach to the {O}seen problem.
\newblock {\em SIAM J. Sci. Comput.}, 28:2095--2113, 2006.

\bibitem{BG71}
K.~M. Brown and W.~B. Gearhart.
\newblock Deflation techniques for the calculation of further solutions of a
  nonlinear system.
\newblock {\em Numerische Mathematik}, 16:334--342, 1971.

\bibitem{BS06}
C.-H. Bruneau and M.~Saad.
\newblock The 2d lid-driven cavity problem revisited.
\newblock {\em Computers $\&$ Fluids}, 35:326--348, 2006.

\bibitem{EPRX20}
C.~Evans, S.~Pollock, L.~Rebholz, and M.~Xiao.
\newblock A proof that {A}nderson acceleration improves the convergence rate in
  linearly converging fixed-point methods (but not in those converging
  quadratically).
\newblock {\em SIAM Journal on Numerical Analysis}, 58:788--810, 2020.

\bibitem{FBF15}
P.~E. Farrell, A.~Birkisson, and S.~W. Funke.
\newblock Deflation techniques for finding distinct solutions of nonlinear
  partial differential equations.
\newblock {\em SIAM J. Sci. Comput.}, 37:A2026--A2045, 2015.

\bibitem{GR86}
V.~Girault and P.-A. Raviart.
\newblock {\em Finite element methods for {Navier--Stokes} equations: {T}heory
  and algorithms}.
\newblock Springer-Verlag, 1986.

\bibitem{laytonbook}
W.~Layton.
\newblock {\em An {I}ntroduction to the {N}umerical {A}nalysis of {V}iscous
  {I}ncompressible {F}lows}.
\newblock SIAM, Philadelphia, 2008.

\bibitem{PRX19}
S.~Pollock, L.~Rebholz, and M.~Xiao.
\newblock Anderson-accelerated convergence of {P}icard iterations for
  incompressible {N}avier-{S}tokes equations.
\newblock {\em SIAM Journal on Numerical Analysis}, 57:615-- 637, 2019.

\bibitem{PS20}
S.~Pollock and H.~Schwartz.
\newblock Benchmarking results for the {N}ewton-{A}nderson method.
\newblock {\em Results in Applied Mathematics}, 8:100095, 2020.

\bibitem{WB02}
K.L. Wong and A.J. Baker.
\newblock A 3d incompressible {N}avier-{S}tokes velocity-vorticity weak form
  finite element algorithm.
\newblock {\em International Journal for Numerical Methods in Fluids},
  38:99--123, 2002.

\bibitem{zhang:scott:vogelius:3D}
S.~Zhang.
\newblock A new family of stable mixed finite elements for the 3d {S}tokes
  equations.
\newblock {\em Math. Comp.}, 74(250):543--554, 2005.

\bibitem{MR2519595}
Shangyou Zhang.
\newblock A family of {$Q_{k+1,k}\times Q_{k,k+1}$} divergence-free finite
  elements on rectangular grids.
\newblock {\em SIAM J. Numer. Anal.}, 47(3):2090--2107, 2009.

\end{thebibliography}
%% -- --------------------------------------------------------------------
%% -- --------------------------------------------------------------------

\end{document}